\documentclass[pdftex]{amsart} 
\usepackage[usenames]{color}
\usepackage{verbatim} 
\usepackage[margin=2cm]{geometry}
\usepackage{amssymb}
\usepackage{multicol}
\usepackage{array}
\usepackage{mathtools}
\usepackage{pdfpages}
\usepackage{color}

\usepackage[english]{babel}

\usepackage{moreverb}
\usepackage{url}
\usepackage{graphicx}
\usepackage{longtable}
\usepackage{lscape}
\usepackage{booktabs}
\usepackage{tabularx}
\usepackage{psfrag}
\usepackage{multicol}
\usepackage{paralist}

\usepackage{tikz}
\usepackage{pgf}
\usetikzlibrary{shapes, intersections, calc, patterns, positioning,
  matrix, arrows, fit, backgrounds, 
  decorations.pathmorphing, decorations.markings}

\newtheorem{thm}{Theorem}

\newtheorem{lemma}[thm]{Lemma}
\newtheorem{prop}[thm]{Proposition}

\newtheorem{cor}[thm]{Corollary}

\newtheorem{defn}[thm]{Definition}

\newtheorem{question}[thm]{Problem}

\theoremstyle{definition}
\newtheorem{exmp}[thm]{Example}
\newtheorem{example}[thm]{Example}
\newtheorem{algorithm}[thm]{Algorithm}

\newtheorem{rmk}[thm]{Remark}

\numberwithin{equation}{section}

\newcommand{\Z}{\mathbf{Z}}

\newcommand{\R}{\mathbf{R}}

\newcommand{\F}{\mathbf{F}}

\newcommand{\chs}[2]{\operatorname{\chi}_{#1}(#2)}

\newcommand{\m}{morphism}
\newcommand{\gen}[1]{{\langle}#1{\rangle}}

\newcommand{\GL}[2]{\mathrm{GL}(#1,\F_{#2})}

\title{Chromatic Numbers of Simplicial Manifolds}

\author{Frank H.~Lutz}
\address{
Technische Universit{\"a}t Berlin\\
Institut f{\"u}r Mathematik, MA 5-2\\
Stra{\ss}e des 17. Juni 136\\
D--10623 Berlin} 
\email{lutz@math.tu-berlin.de}
\urladdr{http://page.math.tu-berlin.de/~lutz/}

\author{Jesper M.~M\o ller}
\address{
K\o benhavns Universitet\\
  Matematisk Institut\\
  Universitetsparken 5\\
  DK--2100 K\o benhavn}
\email{moller@math.ku.dk}
\urladdr{htpp://www.math.ku.dk/~moller}

\thanks{Supported by the Danish National Research Foundation (DNRF)
  through the Centre for Symmetry and Deformation and by VILLUM FONDEN
  through the network for Experimental Mathematics in Number Theory,
  Operator Algebras, and Topology}

\usepackage[bookmarks=true,bookmarksopen=false]{hyperref}
\hypersetup{
   pdftitle = {},
   pdfauthor = {Jesper Michael M??ller},
   pdfpagemode = {UseOutlines},
   pdfstartview = {FitH},
   pdfborder = {0 0 0},
   backref = {true},
   colorlinks = {true},
   urlcolor = {blue},
   citecolor = {blue},
   linkcolor = {blue},
   pdftoolbar = {true},
}

\begin{document}
\maketitle
\tableofcontents

\begin{abstract}
  Higher chromatic numbers $\chi_s$ of simplicial complexes naturally generalize the chromatic number $\chi_1$ of a graph.
  In any fixed dimension $d$,  the $s$-chromatic number $\chi_s$ of $d$-complexes can become arbitrarily large 
  for $s\leq\lceil d/2\rceil$ \cite{bing83,HeisePanagiotouPikhurkoTaraz2014}. In contrast, $\chi_{d+1}=1$, and only little is known
  on  $\chi_s$ for $\lceil d/2\rceil<s\leq d$.
  
  A particular class of $d$-complexes are triangulations of $d$-manifolds. As a consequence of the Map Color Theorem
  for surfaces \cite{Ringel1974}, the $2$-chromatic number of any fixed surface is finite. However, by combining results 
  from the literature, we will see that $\chi_2$ for surfaces becomes arbitrarily large with growing genus. The proof for this is via
  Steiner triple systems and is non-constructive. In particular, up to now, no explicit triangulations of surfaces
  with high $\chi_2$ were known.
  
  We show that orientable surfaces of genus at least~$20$ and non-orientable surfaces of genus at least $26$ 
  have a $2$-chromatic number of at least~$4$.  Via a projective Steiner triple systems, we construct  
  an explicit triangulation of a non-orientable surface of genus $2542$ and with
  face vector $f=(127,8001,5334)$ that has $2$-chromatic number~$5$ or~$6$. 
  We also give orientable examples with $2$-chromatic numbers $5$ and $6$.

  For $3$-dimensional manifolds, an iterated moment curve construction \cite{HeisePanagiotouPikhurkoTaraz2014}
  along with embedding results \cite{bing83} can be used to produce triangulations with arbitrarily large $2$-chromatic
  number, but of tremendous size.  
  Via a topological version of the geometric construction of \cite{HeisePanagiotouPikhurkoTaraz2014}, 
  we obtain a rather small triangulation of the $3$-dimensional sphere $S^3$ 
  with face vector $f=(167,1579,2824,1412)$ and $2$-chromatic number~$5$. 
\end{abstract}

\section{Introduction}
\label{sec:intro}

Let $G=(V,E)$ be a (finite) simple graph with vertex set $V$ and edge
set $E$. The \emph{$1$-chromatic number} $\chi_1(G)$ of~$G$, i.e., the
(standard) chromatic number $\chi(G)$ of~$G$, is the number of colors
needed to color the vertices of $G$ such that no two adjacent vertices
are colored with the same color.  For given $G$, it is NP-hard to
compute $\chi_1(G)$ \cite{garey76,Karp1972}. It is even hard to approximate
$\chi_1(G)$ \cite{Zuckerman2007}, and rather few tools (from algebra, linear algebra,
topology) are available to provide lower bounds for $\chi_1(G)$; see
the surveys \cite{toft95,MatousekZiegler2004} for a discussion. Upper bounds for $\chi_1(G)$
usually are obtained by searching for or constructing explicit colorings of $G$.
 
 
Let $V$ be a finite set, and $E$ be a family of subsets of $V$.
The ordered pair $H=(V,E)$ is called a \emph{hypergraph} with \emph{vertex set} $V$
and \emph{edge set} $E$. If all edges of $H$ have size~2, the hypergraph $H$ is a graph,
while, in general, we can think of $H$ to define a simplicial complex $K$ on the vertex set $V$
with facet list $E$ (not necessarily inclusion-free).

A hypergraph $H=(V,E)$ is \emph{$r$-uniform} if all edges of $H$ are of size $r$, thus defining
a \emph{pure} simplicial complex $K$ of dimension $r-1$. Conversely, every pure
$d$-dimensional simplicial complex can be encoded as a $(d+1)$-uniform hypergraph.

Colorings of hypergraphs generalize colorings of graphs in various ways; see, for example, \cite{KrivelevichSudakov1998}.
The \emph{$s$-chromatic number} $\chi_s(H)$ [$\chi_s(K)$] of a hypergraph $H$ [a finite simplicial complex $K$]
is the minimal number of colors needed to color the vertices $V$ so that
subsets of the sets in $E$ with at most $s$ elements are allowed to be monochrome,
but no $(s+1)$-element subset is monochrome. 

For $s=1$, the vertices of every edge of $E$
are required to be colored with distinct colors, and $\chi_1(H)$ is called the \emph{strong chromatic number} of $H$
[alternatively, $\chi_1(K)$ is the (standard) chromatic number $\chi({\rm skel}_1(K))$ of the $1$-skeleton of $K$]. 
For $s=r-1$, in the case of $r$-uniform hypergraphs,
$\chi_{r-1}(H)$ is called the \emph{weak chromatic number} of $H$. We say that a hypergraph  $H$ [a simplicial complex $K$]
is \emph{$(k,l)$-colorable} if the vertices of $H$ [of $K$] can be colored with $k$ colors so that there
are no monochromatic $(l+1)$-element sets.

\begin{defn}
Let $M$ be a triangulable (closed, connected) manifold of dimension $d$ and $1\leq s\leq d+1$.
Then  
$$\chi_s(M)=\mbox{\rm sup}\{\,\chi_s(K)\,|\,\mbox{\rm $K$ is a triangulation of $M$}\}$$
is the $s$-chromatic number of $M$. 
\end{defn}
The $s$-chromatic numbers of a manifold form a descending sequence
$$\chi_1(M)\geq \chi_2(M)\geq\dots\geq\chi_d(M)\geq\chi_{d+1}(M)=1,$$
where some of these numbers are finite and some are infinite:
\begin{compactitem}
\item[(1)] $\chi_1(S^1)=3$,
\item[(2)] $\chi_1(M^2) =  \left\lfloor \frac{7+\sqrt{49-24\chi_E(M^2)}}{2} \right\rfloor$ \cite{appel&haken:76,Ringel1961,Ringel1974} 
for a surface $M^2$ different from the Klein bottle, 
where $\chi_E(M^2)$ denotes the Euler characteristics of $M^2$, and $\chi_1(\mbox{Klein bottle})=6$, 
\item[(3)] $\chi_1(M^3)=\infty$ for any $3$-manifold $M^3$ \cite{Walkup1970},
\item[(4)] $\chi_2(M^3)=\infty$ for any $3$-manifold $M^3$ (Section~\ref{sec:colASC}),  
\item[(5)] $\chs s{{M}^d} = \infty$ when $M^d$ a triangulable $d$-manifold with $d \geq 3$ and $s \leq \lceil
  d/2 \rceil$ (Section~\ref{sec:colASC}).
\end{compactitem}
Further, we have for surfaces and general $2$-dimensional simplicial complexes:
\begin{compactitem}
\item[(6)] $\chi_1(\mbox{arbitrary surface})=\infty$ \cite{Ringel1961,Ringel1974},
\item[(7)] $\chi_2(\mbox{arbitrary $2$-dimensional simplicial complex})=\infty$ \cite{deBrandesPhelpsRoedl1982,HeisePanagiotouPikhurkoTaraz2014},
\item[(8)] $\chi_2(\mbox{arbitrary surface})=\infty$ (Section~\ref{sec:2-2}).
\end{compactitem}

\medskip

Case (1): Let $C_n$ be a cycle of length $n$. Then $\chi_1(C_n)=\chi(C_n)=2$ for even $n$
and $\chi_1(C_n)=\chi(C_n)=3$ for odd~$n$. Thus, for the $1$-dimensional sphere $S^1$, $\chi_1(S^1)=3.$

Case (2): We have 
\begin{equation}
\chi_1(S^2)=4
\end{equation}
for the $2$-dimensional
sphere $S^2$ by the $4$-color theorem \cite{appel&haken:76}, 
while for surfaces $M^2\neq S^2$ and $M^2$ different from the Klein bottle, 
the map color theorem \cite{ringel&young:68} states
\begin{equation}
  \chi_1(M^2) = 
  \left\lfloor \frac{7+\sqrt{49-24\chi_E(M^2)}}{2}
  \right\rfloor.
\label{eq:map_color}
\end{equation}
In the case of the Klein bottle, 
\begin{equation}
\chi_1(\mbox{Klein bottle})=6.
\end{equation}
Thus, the $1$-chromatic numbers of surfaces are completely determined.

Case (3): Walkup~\cite{Walkup1970} proved that all (closed, connected) $3$-manifolds have neighborly 
triangulations and thus  have $\chi_1=\infty$.  Alternatively, Case (3), and more generally Case (5) 
for $s < \lceil d/2 \rceil$,  follow (for odd $d$; for even $d$ see Section~\ref{sec:colASC}) 
by taking connected sums of the boundaries $\partial CP(m,d+1)$  of neighborly 
cyclic $(d+1)$-polytopes $CP(m,d+1)$  on $m$ vertices
with some triangulation of~$M^d$.

Cases (4) and (7):  Heise, Panagiotou, Pikhurko and Taraz~\cite{HeisePanagiotouPikhurkoTaraz2014} 
provided an inductive geometric construction via the moment curve (this construction was also found independently 
by Jan Kyn\v{c}l  and Josef Cibulka as pointed out to us by Martin Tancer)
to yield $2$-dimensional complexes with arbitrary high $2$-chromatic numbers. 
In Case~(4), PL~topology \cite[Theorem I.2.A]{bing83}  (we are grateful to Karim Adiprasito
for reminding us of this result) is used to extend
geometrically any linearly embedded $2$-dimensional simplicial complex in ${\mathbb R}^3$ 
to a triangulation of the $3$-dimensional ball~$B^3$
that still contains the original $2$-dimensional complex as a subcomplex. If the initial complex 
has high $2$-chromatic number then also the triangulated ball it is contained in. 
By coning off the boundaries of the resulting balls,
we obtain triangulations of $S^3$ with arbitrary high $2$-chromatic number.

Case (5) with $s \leq \lceil d/2 \rceil$ is a generalization of Case (4); see Section~\ref{sec:colASC}.

Case (6) follows as an immediate consequence of Case (2).

Case (7): Steiner triple systems yield examples with arbitrary large $\chi_2$ \cite{deBrandesPhelpsRoedl1982},
                as do the examples of Heise, Panagiotou, Pikhurko and Taraz~\cite{HeisePanagiotouPikhurkoTaraz2014}.

Case (8): Although we might expect that neighborly triangulations of surfaces should have high $2$-chromatic number, 
we show in Section~\ref{sec:bose}, via Bose Steiner triple systems, that there is an infinite series of neighborly triangulations  of surfaces with $\chi_2=3$. 
In contrast, we see in Section~\ref{sec:2-2}  that any Steiner triple system can be extended to a triangulation of a surface
(orientable or non-orientable). Thus, Steiner triple systems with high $\chi_2$ give rise to triangulated surfaces with high $\chi_2$.

In Section~\ref{sec:2-2} we discuss $2$-chromatic numbers of surfaces. In particular, we show that $\chi_2(\mbox{arbitrary surface})=\infty$
(via embeddings of Steiner triple systems with high $\chi_2$) and provide explicit examples of orientable triangulated surfaces 
with $\chi_2=5$ and $\chi_2=6$.
Section~\ref{sec:steiner}  introduces basic concepts of Steiner triple systems
and transversals, which might also be of independent interest in topological design theory.
Section~\ref{sec:bose} discusses neighborly Bose Steiner triple systems with $\chi_2=3$.
In Section~\ref{sec:PGd2}, we use projective Steiner triple systems to construct a triangulation 
of a closed (non-orientable) surface with  $f=(127,8001,5334)$ and $2$-chromatic number $5$ or $6$. 
In Section~\ref{sec:colASC}, by combining results of~\cite{HeisePanagiotouPikhurkoTaraz2014} and~\cite{bing83},
we show that $\chs s {M^d}= \infty$ for  triangulable $d$-manifolds $M^d$  
with $d \geq 3$ and $s \leq \lceil d/2 \rceil$.
In Section~\ref{sec:highspheres}, we provide a topological variant of the geometric iterated moment curve construction of \cite{HeisePanagiotouPikhurkoTaraz2014} to obtain a small triangulation with $f=(167,1579,2824,1412)$ of the $3$-dimension sphere $S^3$ with $2$-chromatic number equal to $5$.

\section{The $2$-chromatic number of surfaces}
\label{sec:2-2}

Let $M^2$ be a surface and $K$ a triangulation of $M^2$ with color classes $1,\dots,\chi_1(K)$, where $\chi_1(K)\leq\chi_1(M^2)$.
If we consecutively merge color classes $1$ with $2$, $3$ with $4$, etc., we obtain a $2$-coloring of $K$ 
with $\left\lceil \frac{\chi_1(K)}2 \right\rceil$ colors as no three vertices of a triangle of~$K$ get colored with the same color.
Thus, for any surface $M^2$,
\begin{equation}
\label{upperbound_chi2}
\chi_2(M^2) \leq \left\lceil \frac{\chi_1(M^2)}2 \right\rceil.
\end{equation}

In any $2$-coloring of a triangulation $K$ of $M^2$,
any individual triangle requires at least two colors,
giving the trivial lower bound
\begin{equation}
\chi_2(M^2)\geq 2.
\end{equation}

In the case of the $2$-sphere $S^2$,
by combining the two bounds, $2\leq\chi_2(S^2)\leq\left\lceil \frac{\chi_1(S^2)}2\right\rceil=2$, we obtain
\begin{equation}
\chi_2(S^2)=2.
\end{equation}

\begin{figure}[htb]
\begin{center}
\includegraphics[width=8.5cm]{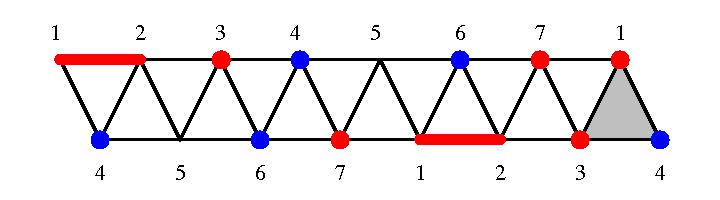}\\
\includegraphics[width=8.5cm]{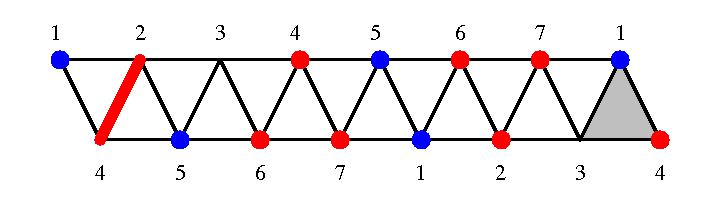}\\
\includegraphics[width=8.5cm]{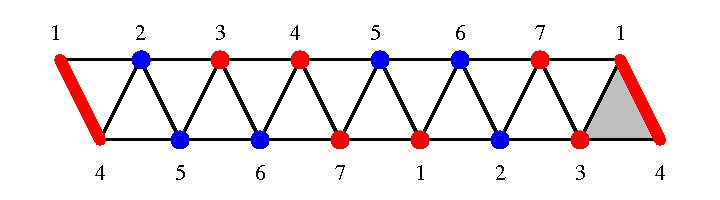}\\
\includegraphics[width=8.5cm]{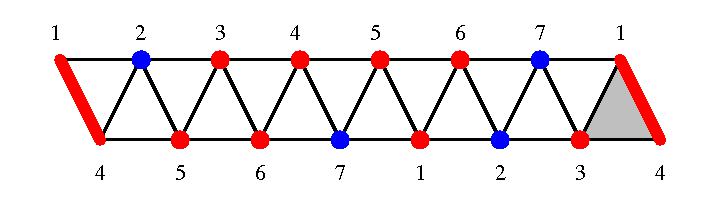}
\end{center}
\caption{The 7-vertex triangulation of the torus with monochromatic triangles in any $(2,2)$-coloring.}
\label{fig:torus}
\end{figure}

\begin{figure}[htb]
\begin{center}
\includegraphics[width=5cm]{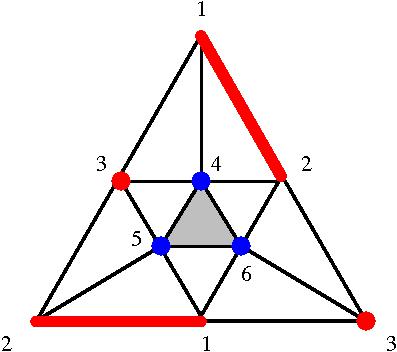}
\end{center}
\caption{The 6-vertex triangulation of the real projective plane with a monochromatic triangle.}
\label{fig:rp_2_6}
\end{figure}

For any surface $M^2\neq S^2$, K\"undgen and Ramamurthi~\mbox{\cite{kundgen_ramamurthi02}}
proved that 
\begin{equation}
\label{eq:KuendgenRamamurthi}
\chi_2(M^2)\geq 3.
\end{equation}

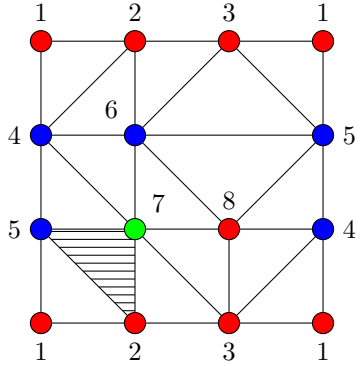
\begin{figure}[t]
  \centering
  
\begin{tikzpicture}[vertex/.style= {shape=circle,  
   fill={#1!100}, minimum size =
  8pt, inner sep =0pt,draw}, vertex/.default=black, scale=.25]  

\coordinate [label={[label distance=4pt] 90: $1$}] (a1) at (0,15);  
\coordinate [label={[label distance=4pt] 90: $2$}] (a2) at (5,15);
\coordinate [label={[label distance=4pt] 90: $3$}] (a3) at (10,15);
\coordinate [label={[label distance=4pt] 90: $1$}] (a4) at (15,15);

\coordinate [label={[label distance=4pt] 180: $4$}] (b1) at (0,10);  
\coordinate [label={[label distance=4pt] 135: $6$}] (b2) at (5,10);
\coordinate [label={[label distance=4pt] 0: $5$}] (b3) at (15,10);

\coordinate [label={[label distance=4pt] 180: $5$}] (c1) at (0,5);  
\coordinate [label={[label distance=4pt] 45: $7$}] (c2) at (5,5);
\coordinate [label={[label distance=4pt] 90: $8$}] (c3) at (10,5);
\coordinate [label={[label distance=4pt] 0: $4$}] (c4) at (15,5);

\coordinate [label={[label distance=4pt] -90: $1$}] (d1) at (0,0);  
\coordinate [label={[label distance=4pt] -90: $2$}] (d2) at (5,0);
\coordinate [label={[label distance=4pt] -90: $3$}] (d3) at (10,0);
\coordinate [label={[label distance=4pt] -90: $1$}] (d4) at (15,0);

\path[pattern=horizontal lines] 
(c1) -- (d2) -- (c2) -- cycle;
\draw (a1)--(a4);   
\draw (b1)--(b3);
\draw (c1)--(c4);   
\draw (d1)--(d4);
\draw (a1)--(d1);
\draw (a2)--(d2);
\draw (a4)--(d4);
\draw (a2)--(b1)--(d3)--(c4);
\draw (a3)--(b2)--(c3)--(b3)--cycle;
\draw (c3)--(d3);  
\draw (c1)--(d2);

\foreach \p in {a1,a2,a3,a4,d1,d2,d3,d4,c3}  \node [vertex=red] at
(\p) {};
\foreach \p in {b1,b2,c1,b3,c4}  \node [vertex=blue] at (\p) {};
\foreach \p in {c2}  \node [vertex=green] at (\p) {};
\end{tikzpicture}

  \caption{A $(3,2)$-coloring of a triangulated Klein bottle (minus a triangle).}
  \label{fig:KB}
\end{figure}

To pave the way for Theorem~\ref{thm:bound4} below, we give an alternative proof of the K\"undgen and Ramamurthi bound~(\ref{eq:KuendgenRamamurthi}):
We first show that the $2$-chromatic number of the unique minimal $7$-vertex triangulation of the torus (see Figure~\ref{fig:torus}) is three 
--- and is still three when we remove (up to symmetry) one of the $14$ triangles of the triangulation. 
Let us assume that two colors (red and blue) suffice when we remove triangle $1\,3\,4$ (shaded in grey in Figure~\ref{fig:torus}) 
from the triangulation. We pick a triangle, say $1\,2\,4$, and color its three vertices $1$, $2$ and $4$ with two colors
such that the triangle is not monochromatic. Thus, two vertices of the triangle get the same color, say, red,
while the third vertex is colored differently, say, in blue.
\begin{itemize}
\item If $1$ and $2$ are colored red and $4$ is colored blue, it follows that also vertex $6$ is colored blue, since otherwise
we would have a monochromatic triangle $1\,2\,6$. The blue edge $4\,6$ then forces the vertices $3$ and $7$ to be red,
which yields monochromatic triangles $1\,3\,7$ and $2\,3\,7$.
\item If $2$ and $4$ are colored red and $1$ is colored blue, also $5$ has to be blue. Via the blue edge
$1\,5$ it then follows that the vertices $6$ and $7$ are red, which yields monochromatic triangles $2\,6\,7$ and $4\,6\,7$.
\item If $1$ and $4$ are colored red and $2$ is colored blue, the color of vertex $3$ remains open,
since we removed triangle $1\,3\,4$ from the triangulation. We therefore consider two subcases.
First, let us assume that $5$ is blue. Then $3$ has to be red (because of the blue edge $2\,5$),
$6$ has to be blue (because of the red edge $3\,4$), and $7$ has to be red (because of the blue edge $2\,6$),
which yields the monochromatic triangle $1\,3\,7$. If instead we color $5$ red, $7$ has to be blue (because of the red edge $4\,5$),
but then $3$ and $6$ are red (because of the blue edge $2\,7$), which gives monochromatic triangles $1\,5\,6$ and $3\,4\,6$.
\end{itemize}

By taking connected sums with the $7$-vertex triangulation of the torus (minus a triangle), we have $\chi_2(M^2)\geq 3$ 
for any surface $M^2$ with $\chi_E(M^2)<0$.
The only two remaining cases then are the real projective plane  ${\mathbb R}{\bf P}^2$ (with $\chi_E({\mathbb R}{\bf P}^2)=1$ 
and the Klein bottle with $\chi_E(\mbox{\rm Klein bottle})=0$.

The vertex-minimal $6$-vertex triangulation ${\mathbb R}{\bf P}^2_6$ (see Figure~\ref{fig:rp_2_6}) of the real projective plane is edge-transitive.
Thus, we can choose the edge $1\,2$ to be red, which forces the vertices $4$ and $5$ to be blue and (via the blue edge $4\,5$) $3$ to be red.
Via the red edge $1\,3$, $6$ has to be blue, which yields a monochromatic triangle $4\,5\,6$. (If we remove the triangle $4\,5\,6$ from the 
triangulation, we obtain a valid $(2,2)$-coloring of the crosscap.)

The Klein bottle has six distinct vertex-minimal $8$-vertex
triangulations, of which three are $(2,2)$-colorable, while the other
three do not admit a $(2,2)$-coloring; see Figure~\ref{fig:KB} for an
$8$-vertex triangulation with a $(3,2)$-coloring --- we leave it to
the reader to check that there is no $(2,2)$-coloring for this
example.

Combining the above with $\chi_1({\mathbb R}{\bf P}^2)=6$ and $\chi_1(\mbox{\rm Klein bottle})=6$,
we get
 \begin{equation}
\chi_2({\mathbb R}{\bf P}^2)=3,
\end{equation}
\begin{equation}
\chi_2(\mbox{\rm Klein bottle})=3,
\end{equation}
and due to K\"undgen and Ramamurthi~\cite{kundgen_ramamurthi02},
\begin{equation}
\chi_2(T^2)=3. 
\end{equation}

Here,
  the unique $7$-vertex triangulation of the torus is the only
  triangulation of the torus with $1$-chromatic
  number~$7$~\cite{Dirac1952}; it is $2$-colorable with $3$ colors.
  All other triangulations of the torus have $1$-chromatic
  number at most~$6$ and thus also are $2$-colorable with $3$ colors.

\begin{rmk}
The $2$-sphere $S^2$, the projective plane ${\mathbb R}{\bf P}^2$, the $2$-torus $T^2$
and the Klein bottle are the only surfaces for which their $2$-chromatic number are known.
\end{rmk}

As noted by K\"undgen and Ramamurthi~\cite{kundgen_ramamurthi02},
there are neighborly triangulations on $19$ vertices with $\chi_2=4$.
We will sharpen this result in Theorem~\ref{thm:bound4} for non-orientable surfaces and give an extension to connected sums.
For the search for $2$-colorings in the case of small vales of $k$ we use a simple brute force strategy.

\begin{algorithm} (Search for $(k,2)$-colorings)\\[1.5mm]
\label{alg:search}
INPUT: Let $K$ be a triangulation of a surface $M^2$ and fix $k$ to search for $(k,2)$-colorings.
\begin{enumerate}
\item Fix a starting triangle of $K$.
\item Fix an initial coloring of the starting triangle. If $k=2$, we consider the three configurations
         with a monochromatic edge colored with color 1 and the third vertex colored with color 2.
         If  $k\geq 3$, we also consider the case that the three vertices of the triangle are colored 
         differently by the three colors $1$, $2$, and~$3$. We write the initial configurations to a stack.

         At any stage, we have a subset of colored vertices and the remaining set of unmarked vertices.
         The unmarked vertices will be split into two sets, the set of \emph{reachable} vertices and the set of \emph{unreachable} vertices.
         An unmarked vertex is reachable if it forms a triangle (in the triangulation $K$) together with two colored vertices.
         A reachable vertex sometimes can be reached via different triangles, which could impose restrictions on the choices
         to color this vertex to avoid monochromatic triangles.  We compute the list of admissible colors for each reachable vertex.
          If a reachable vertex has an empty list of admissible colors,
         we discard this configuration. 

\item We remove a configuration from the stack and choose a reachable vertex for which the cardinality of its list
         of admissible colors is as small as possible. For any choice to color the new vertex, we write the resulting
         configuration (if admissible) back to the stack.         

At any intermediate stage there always are reachable vertices: let $L$ be the subcomplex of the triangulation~$K$ consisting of all triangles 
for which all of their vertices have been colored. Then for any edge on the boundary of $L$ one of its
two incident triangles is in $L$ while the other is not. The latter triangle contains an uncolored, but reachable vertex.

\end{enumerate}
The search terminates either with the conclusion that $K$ does not have any $(k,2)$-colorings 
or outputs a $(k,2)$-coloring. 
\end{algorithm}

We use that $\chi_E=2-2g$ for an orientable surface of genus $g$
and $\chi_E=2-u$ for a non-orientable surface of genus $u$.

\begin{thm}\label{thm:bound4}
Let $M^2$ be an orientable surface of genus $g\geq 20$ (i.e., with $\chi_E(M^2)\leq -38$) or $M^2$ be a non-orientable surface
of genus $u\geq 26$ (i.e., with $\chi_E(M^2)\leq -24$), then
\begin{equation}
\chi_2(M^2)\geq 4.
\end{equation}
\end{thm}

\begin{proof}
We run an implementation of Algorithm~\ref{alg:search} on various triangulations of surfaces from~\cite{Lutz_PAGE} to find examples that do not admit $(3,2)$-colorings.
For each such example, we then rerun the search on the example minus one of its triangles:

The example \texttt{manifold\_cyc\_d2\_n16\_\#10} from~\cite{Lutz_PAGE} is a neighborly cyclic triangulation of the non-orientable surface of genus $u=26$ on $n=16$ vertices with $\chi_E=-24$, which is $(4,2)$-colorable, but not $(3,2)$-colorable. If we remove the triangle $[1,12,14]$ from the triangulation,
there still is no $(3,2)$-coloring. By taking connected sums, we obtain the stated result for non-orientable surfaces
with $\chi_E\leq -24$.

The example \texttt{manifold\_cyc\_d2\_n19\_\#47} from~\cite{Lutz_PAGE} is a neighborly cyclic triangulation of the orientable surface of genus $g=20$ on $n=19$ vertices with $\chi_E=-38$, which is $(4,2)$-colorable, but not $(3,2)$-colorable. If we remove the triangle $[6,8,15]$ from the triangulation,
there still is no $(3,2)$-coloring. By taking connected sums, the stated result follows for orientable surfaces
with $\chi_E\leq -38$.
\end{proof}

\begin{question}
  Is there an orientable surface $M$ of genus $g<20$ or a
  non-orientable surface $M$ of genus $u<26$ with $\chi_2(M^2)\geq 4$?
\end{question}

We tried bistellar flips  \cite{BjoernerLutz2000} to search through the spaces of triangulations
for  various fixed genus surfaces with $g<20$ or $u<26$, but never found an example with  $\chi_2\geq 4$.
For $g\geq 20$ or $u\geq 26$ we also failed to find an example with $\chi_2> 4$ this way. 
However, as we will see next, there are triangulations of orientable and non-orientable surfaces with arbitrary large $2$-chromatic number.

\begin{defn}\label{defn:STS}
  A pure $2$-dimensional simplicial complex $\mathrm{STS}(n)$ with vertex set $V=\{1,\dots,n\}$ is a
  \emph{Steiner triple system} on $V$ if the $1$-skeleton of $\mathrm{STS}(n)$ is the complete
  graph on $V$ and every $1$-simplex lies in a unique $2$-simplex of $\mathrm{STS}(n)$.
\end{defn}

For every $k\geq 3$ there is a Steiner triple system with $\chi_2\geq k$ \cite{deBrandesPhelpsRoedl1982}.
However, the proof of existence in \cite{deBrandesPhelpsRoedl1982} is non-constructive, 
and explicit examples are not known for $k\geq 7$.
Steiner triple systems $\mathrm{STS}(n)$ with $n\geq 7$ have $\chi_2\geq 3$ \cite{rosa70},  
Steiner triple systems with $\chi_2= 4$ can be found, for example, in \cite{deBrandesPhelpsRoedl1982,rosa70}.

For projective Steiner triples systems $\mathrm{PG}(2^d)$ (see Section~\ref{sec:PGd2}), 
$\chs 2{\mathrm{PG}(2^d)} \leq \chs 2{\mathrm{PG}(2^{d+1})} \leq \chs 2{\mathrm{PG}(2^d)}+1$ \cite{rosa70},
where  $\chs 2{\mathrm{PG}(2^5)} =4$ \cite{rosa70} and $\chs 2{\mathrm{PG}(2^6)} =5$ \cite{fugere94}.
By \cite[Corollary 2]{haddad99}, $\chs 2{\mathrm{PG}(2^d)} \to \infty \text{\ for\ } d \to \infty$.

Grannell, Griggs, and \v{S}ir\'a\v{n}~\cite{GrannellGriggsSiran2005} proved that for $n>3$ 
every $\mathrm{STS}(n)$ has both an orientable and a non\-orientable surface embedding in which the triples of the $\mathrm{STS}(n)$ 
appear as triangular faces and there is just one additional large face.

We combine the existence of Steiner triple systems with high $\chi_2$ with the embedding theorem \cite{GrannellGriggsSiran2005}.

\begin{thm} 
$\chi_2(\mbox{arbitrary surface})=\infty$.
\end{thm}

\begin{proof}
For any given $\mathrm{STS}(n)$ on $n$ vertices, Grannell, Griggs, and \v{S}ir\'a\v{n}~\cite{GrannellGriggsSiran2005} 
construct \emph{maximum genus embeddings}  (orientable and non-orientable) that decompose the target surface
into the triangles of the $\mathrm{STS}(n)$ plus one additional polygonal cell that has as its boundary the $\binom{n}{2}$ edges
(of the complete graph $K_n$ as the $1$-skeleton) of the $\mathrm{STS}(n)$. In particular, each vertex of the $\mathrm{STS}(n)$
appears $\binom{n}{2}/n$ times on the boundary of the polygonal cell, but any three consecutive vertices on the boundary differ (as there
are no loops and multiple edges).

We extend the embeddings of \cite{GrannellGriggsSiran2005} to proper triangulations by placing a cycle with $\binom{n}{2}$ additional vertices
in the interior of the polygonal cell and triangulate the annulus between the inner cycle and the boundary cycle in a zigzag way, which gives 
$2\binom{n}{2}$ triangles in the annulus. A triangulation of the inner disk by placing chords requires  $\binom{n}{2}-2$ additional triangles.
 If $\chi_2(\mathrm{STS}(n))=k$, then the resulting triangulated surfaces (orientable and non-orientable)  have 
face vector 
\begin{equation}\label{eq:facevector}
f=\big(n+\textstyle\binom{n}{2},\,5\binom{n}{2}-3,\,\frac{1}{3}\binom{n}{2}+3\binom{n}{2}-2\displaystyle\big),
\end{equation} 
$\chi_E=2-\frac{1}{3}(n-1)(n-3)$ and $\chi_2\geq k$.
\end{proof}

The genus of the resulting surface is $\frac{1}{6}(n-1)(n-3)$ in the orientable and $\frac{1}{3}(n-1)(n-3)$ in the non-orientable case.
The face vector~(\ref{eq:facevector}) of the triangulations can further be improved by not placing an inner cycle vertex for every edge
on the identified boundary of the polygonal cell, but instead have an inner cycle vertex for a sequence of consecutive boundary edges 
that contain no repeated original vertices. Yet, for the simplicity of the construction, we place all $\binom{n}{2}$ new vertices.

\begin{thm}
Every Steiner triple system $\mathrm{STS}(n)$ with $2$-chromatic number $k\geq 3$ has embeddings into triangulations
of the orientable surface of genus $g=\frac{1}{6}(n-1)(n-3)$ and the non-orientable surface of genus $u=\frac{1}{3}(n-1)(n-3)$ 
with \begin{equation}\label{eq:facevector_new}
f=\big(n+\textstyle\binom{n}{2}+1,\,5\binom{n}{2},\,\frac{1}{3}\binom{n}{2}+3\binom{n}{2}\displaystyle\big)
\end{equation} 
$\chi_E=2-\frac{1}{3}(n-1)(n-3)$ and $\chi_2=k$.
\end{thm}

\begin{proof}
Instead of adding chords to triangulate the inner disc, we place a central vertex in the interior of the inner cycle and add the cone with respect to the cycle.
This yields  triangulations that have one vertex, three edges and two triangles extra and thus have
 $f=\big(n+\textstyle\binom{n}{2}+1,\,5\binom{n}{2},\,\frac{1}{3}\binom{n}{2}+3\binom{n}{2}\displaystyle\big)$.

Let the $n$ vertices (with repetitions) on the outer cycle be colored with $k$ colors such that the Steiner triple system $\mathrm{STS}(n)$ 
has no monochromatic triangle. Then only two of the colors $1,\dots,k$ are needed to color the inner cycle:
For this, we go along the inner cycle of the zigzag collar and color the vertices with color~1 until a neighboring vertex on the outer cycle has color~1 as well.
In this case, we switch to color~2 until an outer vertex is colored with color~2, then switch back to color~1, etc.
For the cone vertex we are then free to use color~3.
\end{proof}

The following triangulation scheme provides an explicit algorithmic version of the embedding procedure of \cite{GrannellGriggsSiran2005}
along with the completion to a triangulation.

\begin{algorithm}\label{alg:embedding}
(Triangulation of an orientable maximum genus embedding of a Steiner triple system)\\[1.5mm]
INPUT: Let $\mathrm{STS}(n)$ be a Steiner triple system on $n$ vertices $1,\dots,n$.
\begin{enumerate}
\item Start with an embedding of the star of the vertex $1$ in $\mathrm{STS}(n)$ into the $2$-sphere $S^2$. 
         For this, there are  $\frac{1}{2}(\frac{n-1}{2}!)2^{\frac{n-1}{2}}$ choices. 
\item Read off the boundary cycle of the embedded star. 
         Let the triangles of the star be `black' and the outside region  (a disk with identifications on the boundary) be `white'.
\item Proceed iteratively for the remaining triangles of the $\mathrm{STS}(n)$ (in any order). For each picked triangle $uvw$,
         add a handle to the white area. Choose occurrences of the vertices $u$, $v$, $w$ on the boundary cycle and glue in the triangle $uvw$
         by routing it via the newly added handle in a twisted way (see Figure~\ref{fig:embed_sts7}), such that after the addition of $uvw$, the white outside area is a disk again.
         For the chosen $u$, $v$, $w$, let $uAvBwC$ be the boundary cycle  before the addition of $uvw$.
         By the addition of the handle and $uvw$, the boundary cycle is then expanded to $uAvwCuvBw$, by inserting the edges $vw$, $uv$, and $wu$ 
         and by traversing $C$ before~$B$.
\item Once all remaining triangles of $\mathrm{STS}(n)$ are added, the boundary cycle contains $\binom{n}{2}$ vertices (multiple copies) 
         and all $\binom{n}{2}$  edges 
         of the complete graph $K_n$ (i.e., the boundary cycle is an Eulerian cycle through $K_n$).
\item Place a second cycle with $\binom{n}{2}$ new vertices in the interior of the white disk and triangulate the annulus between the two cycles in a zigzag way.
\item Finally, triangulate the inner polygon, e.g., by adding the cone of an additional vertex as apex with respect to all edges on the inner polygon. 
\end{enumerate}
\end{algorithm}

\begin{figure}[t]
\begin{center}
\includegraphics[width=6cm]{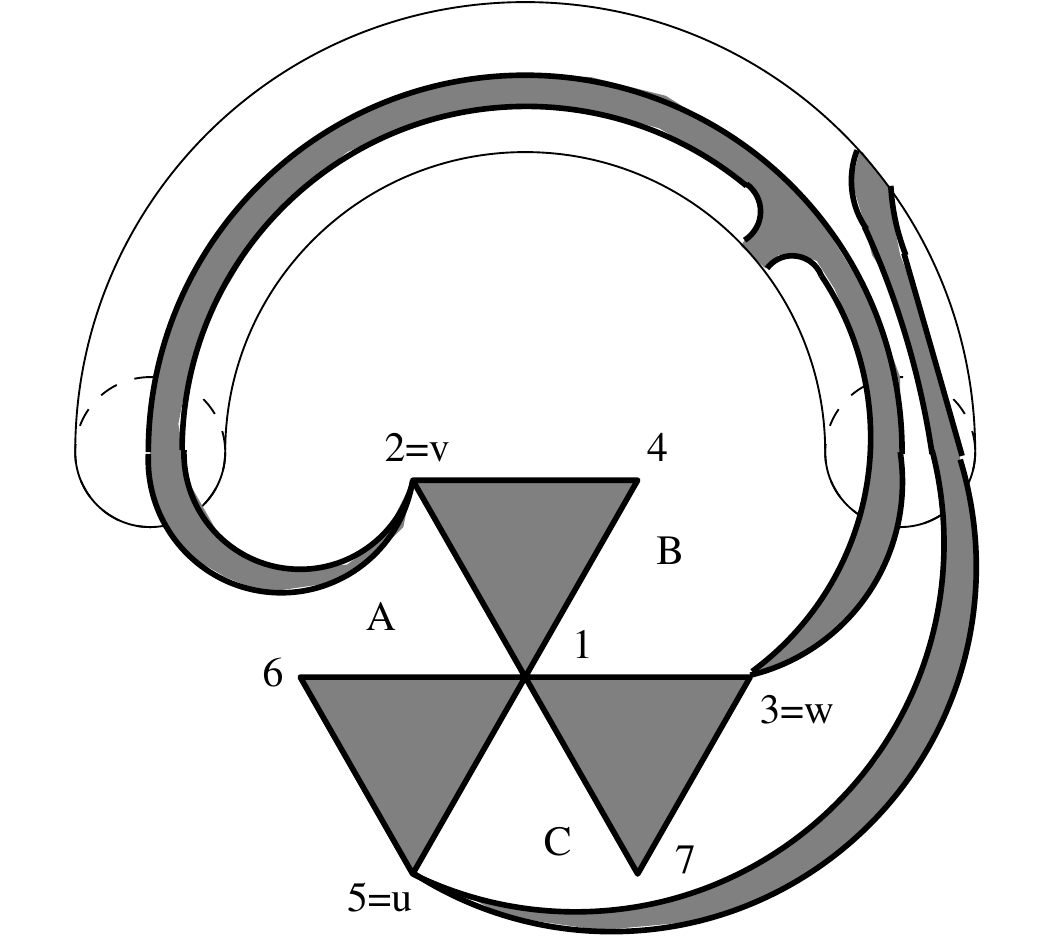}
\end{center}
\caption{The addition of a triangle to the embedded vertex star of $\mathrm{STS}(7)$.}
\label{fig:embed_sts7}
\end{figure}

For non-orientable embeddings, Step (3) of Algorithm~\ref{alg:embedding} can simply be modified 
by not adding a handle for the final triangle, but a twisted handle (two crosscaps) instead \cite{GrannellGriggsSiran2005}.
The cycle $uAvBwC$ is then expanded to $uAvw\overline{B}vu\overline{C}w$, where $\overline{B}$ and $\overline{C}$
denote that $B$ and $C$ are traversed in the opposite direction, respectively.

\begin{example}
The seven upper triangles $124$, $137$, $156$, $235$, $267$, $346$, $457$
of Figure~\ref{fig:torus} form (up to isomophism) the unique Steiner triple system $\mathrm{STS}(7)$ on $7$ vertices.
Once an embedding for the star of vertex $1$, consisting of the triangles $124$, $137$, $156$, is chosen, 
four handles are added along with the remaining four triangles,  and the boundary cycle is expanded each time. Initially, we have $124137156$ as boundary cycle,
which is expanded as follows:
\begin{equation*}
\begin{array}{llll}
235: &  1\fbox{2}41\fbox{3}71\fbox{5}6                      & \longrightarrow &  1\fbox{2}\fbox{3}71\fbox{5}\fbox{2}41\fbox{3}\fbox{5}6, \\[2mm]
267: &  1\fbox{2}3\fbox{7}1524135\fbox{6}                & \longrightarrow &  1\fbox{2}\fbox{7}1524135\fbox{6}\fbox{2}3\fbox{7}\fbox{6}, \\[2mm]
346: &  127152\fbox{4}1\fbox{3}5\fbox{6}2376          & \longrightarrow &  127152\fbox{4}\fbox{3}5\fbox{6}\fbox{4}1\fbox{3}\fbox{6}2376, \\[2mm]
457: &  12\fbox{7}1\fbox{5}2\fbox{4}35641362376    & \longrightarrow &  12\fbox{7}\fbox{5}2\fbox{4}\fbox{7}1\fbox{5}\fbox{4}35641362376.
\end{array}
\end{equation*}
The interior of the resulting cycle 127524715435641362376 is triangulated and these `white' triangles together with the `black' triangles
of the Steiner triple system give a triangulation of the orientable genus $4$ surface that contains the Steiner triple system.
\end{example}

The smallest known Steiner triple system with $\chi_2=5$ is $\mathrm{PG}(64)$ on $63$ vertices,
yielding triangulated surfaces (orientable and non-orientable) with $f=(2017,9765,6510)$ that have $\chi_2=5$.

\begin{thm} 
The orientable surface of genus $g=620$ and the non-orientable surface of genus $u=1240$ have 
triangulations with $f=(2017,9765,6510)$ and $\chi_2=5$.
\end{thm}

The smallest known Steiner triple system with $\chi_2=6$ is $\mathrm{AG}(5,3)$ on $3^5=243$ vertices \cite{bruen98}.

\begin{thm} 
The orientable surface of genus $g=9680$ and the non-orientable surface of genus $u=19360$ have 
triangulations with $f=(29647,147015,98010)$ and $\chi_2=6$.
\end{thm}

For the orientable cases, list of facets \texttt{PG64\_n2017\_o1\_g620} and \texttt{AG\_5\_3\_n29647\_o1\_g9680}
can be found online at~\cite{BenedettiLutz_LIBRARY}.

As mentioned above, \emph{maximum genus embeddings} have genus $\frac{1}{6}(n-1)(n-3)$ in the orientable 
and genus $\frac{1}{3}(n-1)(n-3)$ in the non-orientable case.

\emph{Minimum genus embeddings} of Steiner triple systems are obtained by complementing
a given Steiner triple system $\mathrm{STS}(n)$ with a transversal and isomorphic Steiner triple system $\mathrm{STS}(n)^T$
on the same number of vertices so that the union $\mathrm{STS}(n)\cup \mathrm{STS}(n)^T$ triangulates a surface
with face vector
\begin{equation}
f=\big(n,\textstyle\binom{n}{2},\frac{2}{3}\binom{n}{2}\displaystyle\big)
\end{equation}
and $\chi_E=2-\frac{1}{6}(n-3)(n-4)$. These bi-embeddings give orientable genus $\frac{1}{12}(n-3)(n-4)$ surfaces
or non-orientable genus $\frac{1}{6}(n-3)(n-4)$ surfaces; cf.\  \cite{grannell1998,rifa2014,soloveva2007}.
However, transversals need not exist in general.

The following section provides background and results on Steiner triple systems and transversals.
In Section~\ref{sec:bose}, we show that the examples in the infinite series of neighborly Bose Steiner 
surfaces all have $\chi_2=3$. In Section~\ref{sec:PGd2}, we will see that the projective Steiner triple 
system $\mathrm{PG}(128)$ with $\chi_2(\mathrm{PG}(128)) \in \{5,6\}$ admits a non-orientable transversal $\mathrm{PG}(128)^T$, 
thus yielding a non-orientable surface $\mathrm{PG}(128)\cup\mathrm{PG}(128)^T$
of genus $2542$ with $f=(127,8001,5334)$ and $2$-chromatic number $\chi_2(\mathrm{PG}(128)\cup\mathrm{PG}(128)^T)\in \{5,6\}$.

\section{Steiner triple systems and chromatic numbers of Steiner surfaces}
\label{sec:steiner}

This section contains a systematic study of Steiner triple systems and \emph{Steiner surfaces}, 
in particular, with the aim to clarify some inconsistencies in the literature (see the end of Section~\ref{sec:PGd2} for further comments).
We say that two Steiner triple systems on a set
$V$ are \emph{transversal} (\emph{orientable transversal}) if their union is an
(orientable) triangulated surface with vertex set $V$. We develop
criteria for transversality and orientability and show that
transversals to the Steiner triple system $\mu$ are indexed by the
double coset $\Sigma(\mu) \backslash \Sigma(V) / \Sigma(\mu)$ where
$\Sigma(\mu)$ is the automorphism group of~$\mu$ and $\Sigma(V)$ the
automorphism group of the set $V$. 
In Section~\ref{sec:PGd2} we use this characterization to carry out systematic searches 
for orientable) transversals to Steiner triple systems. 

Let $V$ be a finite set with an odd number $n = 2m+1$ of elements. Let $\mathrm{STS}(V)$ be the set of all Steiner triple systems on $V$. 

\begin{defn}\label{defn:STSgroup}
  A \emph{Steiner quasigroup} on $V$ is a binary operation $\mu \colon V
  \times V \to V \colon (x,y) \to x \cdot y$ such that
  \begin{itemize}
  \item $x \cdot x =x$,  
  \item $x \cdot y=y \cdot x$,
  \item $x\cdot (x \cdot y)=y$
  \end{itemize}
for all $x,y \in V$.
\end{defn}

We shall use the following notation:
\begin{itemize}

\item $\Sigma(V)$ is the group of all permutations of $V$.

\item $S(\mu) = \{
\{x,y,\mu(x,y) \} \mid x,y \in V,\; x \neq y \}$ is the Steiner triple
system associated to the Steiner quasigroup~$\mu$. (We shall often not distinguish between 
a Steiner quasigroup~$\mu$ on $V$ and the associated Steiner triple system $S(\mu)$ on $V$.)

\item If $\mu \in \mathrm{STS}(V)$ and $T \in \Sigma(V)$ then $\mu^T
  \in \mathrm{STS}(V)$ is the Steiner quasigroup given by
  $\mu^T(x,y)=\mu(x^{T^{-1}},x^{T^{-1}})^T$ for all $x,y \in V$.

\item $\Sigma(\mu)=\{ a \in \Sigma(V) \mid 
\forall x,y \in V \colon \mu(x^a,y^a) = \mu(x,y)^a \}$ is the
symmetry group of the Steiner quasigroup $\mu$.

\end{itemize}

There is a right action of the symmetry group of $V$ on the set of
Steiner triple systems on $V$,
\begin{equation}
  \label{eq:STSxSigmaV}
 \mathrm{STS}(V) \times \Sigma(V) \to \mathrm{STS}(V),
\end{equation}
taking $(\mu,T)$ to $\mu^T$. The isotropy subgroup at $\mu$ is
$\Sigma(\mu)$: $T \in \Sigma(\mu) \iff \mu^T=\mu$. We have
$\mu^T(x^T,y^T) = \mu(x,y)^T$ and $S(\mu)^T = S(\mu^T)$ for all $\mu
\in \mathrm{STS}(V)$ and all $T \in \Sigma(V)$.

We note the following well-known result from design theory.  

\begin{thm}
   $\text{$\mathrm{STS}(V) \neq \emptyset$} \iff 
   \text{$n \equiv 1 \bmod 6$ or $n \equiv 3 \bmod 6$}$.
\end{thm}

Suppose that $S^-$ and $S^+$ are two Steiner triple systems on $V$. We
say that $S^-$ and $S^+$ are disjoint when $S^- \cap S^+ = \emptyset$.
Let $S^- \cup S^+$ denote the abstract simplicial complex generated by
their union.  If $S^-$ and $S^+$ are disjoint, $S^- \cup S^+$ is a
$1$-neighborly pseudo-surface (as every $1$-simplex lies in exactly
two $2$-simplices) of Euler characteristic
\begin{equation}\label{eq:eulerSL}
  \chi(S^- \cup S^+) =
   n-\binom n2 + \frac{2}{3}\binom n2 = -\frac{n(n-7)}{6}.
\end{equation}
The Euler characteristic is odd when $n \equiv 1 \bmod 12$ or $n
\equiv 9 \bmod 12$ and even when $n \equiv 3 \bmod 12$ or $n \equiv 7
\bmod 12$.  

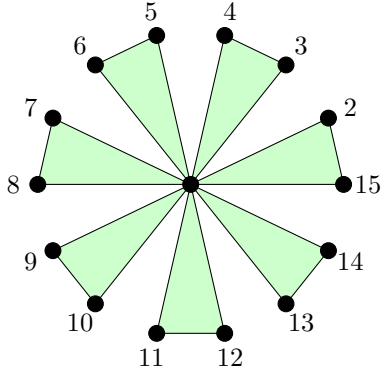
\begin{figure}[t]
  \centering

\begin{tikzpicture}[vertex/.style= {shape=circle,  
   fill={#1!100}, minimum size =
  6pt, inner sep =0pt,draw}, vertex/.default=black, scale=1.3]
 \pgfmathsetmacro{\m}{7}  //\m is not used here
 \pgfmathsetmacro{\r}{2*sin(360/7)};
 \pgfmathsetmacro{\a}{360/7}
  \foreach \d in {1,...,7} {
    \draw[fill=green!20] 
    (0,0) -- (\d*\a : \r ) node[vertex] {}  
    -- (1/2*\a+\d*\a : \r ) node[vertex] {} -- cycle;};
   \foreach \d/\l in {1/3,2/5,3/7,4/9,5/11,6/13,7/15}
   \node at ($(\d*\a : \r+0.25)$)  {$\l$};
    \foreach \d/\l in {1/2,2/4,3/6,4/8,5/10,6/12,7/14}
   \node at ($(-1/2*\a+\d*\a : \r+0.25)$)  {$\l$};
   \node[vertex] at (0,0) {};
\end{tikzpicture}

  \caption{A vertex star in a Steiner triple system on $n=15$ vertices.}
  \label{fig:vertexlink}
\end{figure}
 
\begin{defn}\label{defn:steinertria} 
  An {\em (orientable) Steiner surface\/} on $V$ is a pair, $S^-, S^+
  \in \mathrm{STS}(V)$, of disjoint Steiner triple systems on $V$ so
  that $S^- \cup S^+$ is an (orientable) triangulated surface. 

  An (orientable) {\em transversal\/} to the Steiner triple system $S
  \in \mathrm{STS}(V)$ is a symmetry $T \in \Sigma(V)$ such that
  $(S,S^T)$ is an (orientable) Steiner surface on $V$.
\end{defn}

The genus of a Steiner surface on $V$ is $\frac{1}{6}(n-4)(n-3)$ if it
is nonorientable and $\frac{1}{12}(n-4)(n-3)$ if it is orientable.
Orientable Steiner surfaces exist only for $n \equiv 3,7 \bmod 12$.

Suppose that $S^-$ and $S^+$ are two Steiner triple systems on $V$ and
$\mu^-$ and $\mu^+$ the corresponding Steiner quasigroups on $V$.  The
transition map
\begin{equation}
  \label{eq:transition}
  \sigma \colon \mathrm{STS}(V) \times V \times \mathrm{STS}(V) 
  \to \Sigma(V), \quad y^{\sigma(\mu^-,x,\mu^+)} =
  \mu^-(x,\mu^+(x,y)), \qquad x,y \in V, \quad \mu^-,\mu^+ \in
  \mathrm{STS}(V),
\end{equation}
records the transition from $\mu^+$ to $\mu^-$ in that
$\mu^+(x,y)^{\sigma}= \mu^-(x,y)$ and $\mu^-(x,y^\sigma)=\mu^+(x,y)$
where $\sigma =\sigma(\mu^-,x,\mu^+)$. Note also that
$\mu^+(x,y^\sigma)^\sigma = \mu^+(x,y)$ and
$\mu^+(x,y^{\sigma^j})^{\sigma^{j}} = \mu^+(x,y)$ for all natural
numbers $j$ by induction.

\begin{prop}\label{prop:Mmupm}
    $\text{$(\mu^-,\mu^+)$ is a Steiner surface on $V$} \iff
    \forall x \in V \colon 
    \text{$\sigma(\mu^-,x,\mu^+)$ has cycle structure $1^1m^2$}$.
\end{prop}
\begin{proof}
  Assume that $\sigma(\mu^-,x,\mu^+)$ has cycle structure $1^1m^2$.
  Suppose that $x,y \in V$, $x \neq y$, and $\mu^-(x,y) = \mu^+(x,y)$.
  Then $\sigma(\mu^-,x,\mu^+)$ fixes $\mu^-(x,y)$ so $x=\mu^-(x,y)$.
  This contradiction shows that $\mu^-$ and $\mu^+$ define disjoint
  Steiner triple systems.  Now consider the link at $x$ in the
  simplicial complex $S(\mu^-) \cup S(\mu^+)$. The orbits of
  $\sigma=\sigma(\mu^-,x,\mu^+)$ through $y$ and $\mu^+(x,y)$ consist
  of the $m$ distinct points $y,y^\sigma,\ldots,y^{\sigma^{m-1}}$ and
  $\mu^+(x,y),\mu^+(x,y)^\sigma,\ldots,\mu^+(x,y)^{\sigma^{m-1}}$,
  respectively. Observe that the permutations $y \to
  \mu^+(x,y)^{\sigma^j}$ fix $x$ and only $x$ as they are conjugate to
  one of the permutations $y \to \mu^\pm(x,y)$. Thus the orbits of
  $\sigma$ through $y$ and $\mu^+(x,y)$ are disjoint
  (Figure~\ref{fig:localorient}). It follows that the $2m$ vertices in
  the walk
  \begin{equation*}
    y \to \mu^+(x,y) \to y^\sigma \to \cdots \to
    \mu^+(x,y^{\sigma^{m-1}}) = \mu^+(x,y)^{\sigma} \to y^{\sigma^{m-1}}
  \end{equation*}
  are all distinct. We conclude that the link at $x$ is the cyclic
  triangulation of $S^1$ on $2m$ vertices.

Conversely, assume that $(\mu^-,\mu^+)$ is a Steiner surface.  Let
$x,y \in V$, $x \neq y$.  Since $x$ belongs to $m$ triples from
$\mu^\pm$ the link at $x$ is a $2m$-cycle
   \begin{equation*}
     y \to \mu^+(x,y) \to y^\sigma \to \mu^+(x,y)^\sigma \to \cdots
     \to y^{\sigma^m} = y
   \end{equation*}
   starting at any point $y \neq x$. Thus the orbit of $\sigma$
   through $y$ has size $m$.
\end{proof}

Given an action of a group $H$ on $V$, a {\em transversal\/} is a subset
$V_H \subseteq V$ containing exactly one element from each $H$-orbit
in $V$.

\begin{lemma}\label{lemma:shiftA}
  Suppose that $A \leq \Sigma(\mu^-) \cap \Sigma(\mu^+)$. Then
  $(\mu^-,\mu^+)$ is a Steiner surface if and only if
  \begin{equation*}
    \forall y \in V_A \colon \,\,
    \text{$\sigma(\mu^-,x,\mu^+)$ has cycle structure $1^1m^2$},
  \end{equation*}
  where $V_A \subseteq V$ is transversal to the action of $A$ on $V$.
\end{lemma}
\begin{proof}
  For any $a \in \Sigma(\mu^-,\mu^+)$,
  $\sigma(\mu^-,x^a,\mu^+) = \sigma(\mu^-,x,\mu^+)^a$ for all $x \in
  V$.
\end{proof}

\begin{figure}[t]
\centering

\begin{tikzpicture}[vertex/.style= {shape=circle,  
   fill={#1!100}, minimum size =
  6pt, inner sep =0pt,draw}, vertex/.default=black, scale=1.25]

\begin{scope}[xshift=-5cm,>= stealth',scale=1.3]
  \node (A) at (0,0.75) {$\tau_u^+$};
  \draw[fill=green!20] ($(A)+(90:1)$) node[above] {$\mu^+(x,y)$} -- 
  ($(A)+(210:1)$) node[left] {$x$} -- node[midway] {$>$} 
  ($(A)+(330:1)$) node[right] {$y$} -- cycle;

  \draw[->] 
  ($(A)+(45:.25)+(0,-.25)$) arc[radius=.25, start angle=-45, end angle=235];
  \node at ($(A)+(1.5,.5)$) {$x \in \sigma^+(\mu^+(x,y))$}; 
  
  \node (B) at (0,-0.75) {$\tau_v^-$};
  \draw[fill=yellow!20] ($(B)+(-90:1)$) node[below] {$\mu^-(x,y)$} -- 
  ($(B)+(-210:1)$) node[left] {$x$} -- node[midway] {$<$} 
  ($(B)+(-330:1)$) node[right] {$y$} -- cycle;

  \draw[->] 
  ($(B)+(45:.25)+(0,-.4)$) arc[radius=.25, start angle=-45, end
  angle=235];
   \node at ($(B)+(1.5,-0.5)$) {$x \in \sigma^+(\mu^-(x,y))$}; 
\end{scope}

  \pgfmathsetmacro{\m}{7}  //\m is not used here
 \pgfmathsetmacro{\r}{2*sin(360/7)};
 \pgfmathsetmacro{\a}{360/7}
  \foreach \d in {1,...,7} {
    \draw[fill=green!20!white] 
    (0,0) -- (\d*\a : \r ) node[vertex] {}  
     -- (1/2*\a+\d*\a : \r )  node[vertex] {} -- cycle;};
   \foreach \d/\l in {1/3,2/5,3/7,4/9,5/11,6/13,7/15}
   \node at ($(\d*\a : \r+0.25)$)  {$\l$};
    \foreach \d/\l in {1/2,2/4,3/6,4/8,5/10,6/12,7/14}
   \node at ($(-1/2*\a+\d*\a : \r+0.25)$)  {$\l$};
   \node[vertex] at (0,0) {};
    \foreach \d in {1,...,7} {
    \draw[fill=yellow!20] 
    (0,0) -- (1/2*\a+\d*\a : \r ) node[vertex] {}  
     -- (\a+\d*\a : \r )  node[vertex] {} -- cycle;};
   \node[text width = 6.5cm, align = justify, anchor= north west] at
   (1.7,2)
   {\small $\mu^-(1,-) = (1)(2,3)(4,5)\cdots(14,15) \\
     \mu^+(1,-) = (1)(3,4)(5,6)\cdots(13,14)(15,2) \\
     \sigma(\mu^-,1,\mu^+) = (1)(14,12,\ldots,2)(3,5,\ldots,15)$};
   \draw[thick,->] 
   ($(\a : \r+0.45)$) arc[start angle = \a, delta angle =
   155, radius = 2cm];
   \node at  ($(90 : \r+0.75)$) {$(3,5,\cdots,15)$};
   \draw[thick,->] 
   ($(1/2*\a : \r+0.45)$) arc[start angle = 1/2*\a, delta angle =
   -155, radius = 2cm];
    \node at  ($(-90 : \r+0.75)$) {$(14,12,\cdots,2)$};
\end{tikzpicture}

  \caption{Local orientation.}
  \label{fig:localorient}
\end{figure}
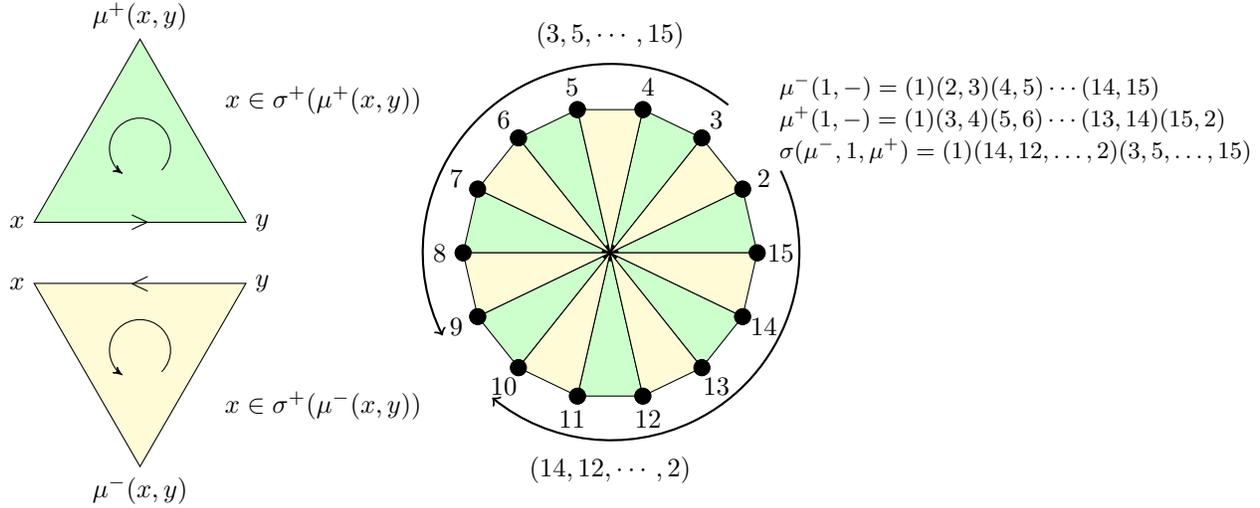

\begin{defn}\label{defn:steinertriaorient}
 Suppose  $(\mu^-,\mu^+)$ is a Steiner surface on $V$.
 \begin{itemize}
 \item A local orientation for $(\mu^-,\mu^+)$ at $x \in V$ is a size
   $m$ orbit $\sigma^+(x)$ 
   for the transition permutation $\sigma(\mu^-,x,\mu^+)$.
 \item An orientation for $(\mu^-,\mu^+)$ is a function, $x \to
   \sigma^+(x)$, that to every $x \in V$ associates a local orientation
   $\sigma^+(x)$ at $x$ in such a way that
  \begin{equation*}
    \forall x,y \in V  \colon \,\,
    x \in \sigma^+(\mu^+(x,y)) \iff   x \in \sigma^+(\mu^-(x,y)).
  \end{equation*}
\item An auto\m\ $f$ of a Steiner surface $(\mu^-,\mu^+)$ with
  orientation $x \to \sigma^+(x)$ is orientation preserving if
  $f(\sigma^+(x)) = \sigma^+(f(x))$ for all $x \in V$.
 \end{itemize}
  \end{defn}

\begin{prop}\label{prop:Mmupmotient}
   Suppose  $(\mu^-,\mu^+)$ is a Steiner surface on $V$.
   \begin{equation*}
     \text{The triangulated surface $S(\mu^-) \cup S(\mu^+)$ admits 
       an orientation} \iff 
     \text{$(\mu^-,\mu^+)$ admits an orientation.}
   \end{equation*}
\end{prop}
\begin{proof}
  Let $V \ni x \to \sigma^+(x)$ be an orientation of the Steiner
  surface $(\mu^-,\mu^+)$. For every $x \in V$, let $\sigma^-(x)$
  denote the size $m$ orbit distinct from $\sigma^+(x)$ of the
  permutation $\sigma(\mu^-,x,\mu^+)$.  Whenever $x$ and $y$ are
  distinct points of $V$ and $x \in \sigma^\pm(y)$, declare
  $(x,\mu^\pm(x,y),y)$ to be a positive permutation of the vertices of
  the $2$-simplex $\{x,\mu^\pm(x,y),y\} \in S(\mu^-) \cup S(\mu^+)$.
  This is a consistent orientation of the $2$-simplices of the
  triangulated surface $S(\mu^-) \cup S(\mu^+)$.
  
  If the surface $S(\mu^-) \cup S(\mu^+)$ is
  oriented then the induced orientation of each vertex link picks out
  one of the $m$-cycles of that vertex link.
\end{proof}

Let $a \in \Sigma(\mu^-)$ be a symmetry of the Steiner quasigroup
$\mu^-$ and let $x,y \in V$. From 
\begin{equation*}
  y^{a\sigma(\mu^-,x^a,\mu^+)} = \mu^-(x^a,\mu^+(x^a,y^a))
  = \mu^-(x^a,(\mu^+)^{a^{-1}}(x,y)^a)
  = \mu^-(x^a,(\mu^+)^{a^{-1}}(x,y))^a
  = y^{\sigma(\mu^-,x,(\mu^+)^{a^{-1}})a} 
\end{equation*}
we see that $\sigma(\mu^-,x^a,\mu^+) = \sigma(\mu^-,x,{}^a\mu^+)^a$
where we write ${}^a\mu^+ = \big(\mu^+\big)^{a^{-1}}$ for the left
action of $\Sigma(\mu^-)$ on $\mathrm{STS}(V)$.  Thus $(\mu^-,\mu^+)$
is an (orientable) Steiner surface if and only if $(\mu^-,{}^a\mu^+)$
is an (orientable) Steiner surface. (If $x \to \sigma^+(x)$ is a local
orientation for $(\mu^-,\mu^+)$ then $ x \to \sigma(x^a)^{a^{-1}}$ is
a local orientation for $(\mu^-,{}^a\mu^+)$.)

The following lemma shows that a Steiner surface is orientable if and
only if the local orientation at any point propagates coherently to
local orientations at all other points.

\begin{lemma}\label{lemma:propagate}
Suppose that $(\mu^-,\mu^+)$ is a Steiner surface on $V$. 
Let $v$ be a vertex in $V$ and $\sigma^+(v)$ a local orientation at
$v$. Then $(\mu^-,\mu^+)$ is orientable if and only if 
\begin{equation*}
  \forall u_1,u_2 \in \sigma^+(v) \forall 
  x \in v^{\gen{\sigma(\mu^-,u_1,\mu^+)}}
  \cap v^{\gen{\sigma(\mu^-,u_2,\mu^+)}} \colon 
  u_1^{\gen{\sigma(\mu^-,x,\mu^+)}}   = u_2^{\gen{\sigma(\mu^-,x,\mu^+)}}
\end{equation*}
or, equivalently, if and only if the projection onto the first
coordinate
\begin{equation*}
  \bigcup_{u \in \sigma^+(v)} \{(x,u^{\gen{\sigma(\mu^-,x,\mu^+)}}) \mid
  x \in v^{\gen{\sigma(\mu^-,u,\mu^+)}} \} \to V
\end{equation*}
is a bijection.
\end{lemma}
\begin{proof}
  The condition of Definition~\ref{defn:steinertriaorient} for an
  orientation is that
\begin{equation*}
  u \in \sigma^+(v) \iff u \in \sigma^+(v^{\sigma(u)})
\end{equation*}
for all distinct vertices $u,v \in V$. (Put $u=x$ and $v=\mu^+(x,y)$
so that $\mu^-(x,y) = \mu^+(x,y)^{\sigma(x)} = v^{\sigma(x)} =
v^{\sigma(u)}$.)  Repeated application of this gives
\begin{equation*}
  u \in \sigma^+(v) \iff u \in \sigma^+(v^{\sigma(u)^i})
\end{equation*}
or, equivalently,
\begin{equation*}
  u \in \sigma^+(v) \iff 
  \forall x \in v^{\gen{\sigma(u)}} \colon u \in \sigma^+(x).
\end{equation*}
This means that $\sigma^+(x)$ is the orbit through $u$ of $\sigma(x)$
for all $x$ in the orbit through $v$ of $\sigma(u)$.
\end{proof}

\begin{lemma}\label{lemma:LAmu}
  Suppose that $A \leq \Sigma(\mu^-,\mu^+)$ for some $A \leq
  \Sigma(V)$. Let $s^+ \in \Sigma(V)$ be a fixed-point free
  permutation of $V$ centralizing $A$. Then $V \ni x \to
  s^+(x)^{\gen{\sigma(\mu^-,x,\mu^+)}}$ is an orientation for the
  Steiner surface $(\mu^-,\mu^+)$ if and only if
  \begin{equation*}
    \forall x \in V\,\, \forall y \in V_A \colon \,\,
    x \in s^+(\mu^+(x,y))^{\gen{\sigma(\mu^-,x,\mu^+)}} \iff 
    x \in s^+(\mu^-(x,y))^{\gen{\sigma(\mu^-,x,\mu^+)}},
  \end{equation*}
  where $V_A \subseteq V$ is transversal to the action of
  $A$ on $V$.
\end{lemma}
\begin{proof}
  Put $\sigma(x) = \sigma(\mu^-,x,\mu^+)$ and $\sigma^+(x) =
  s^+(x)^{\gen{\sigma(x)}}$.  From the proof of
  Lemma~\ref{lemma:shiftA}, $\sigma(x^a) = \sigma(x)^a$ for all $a \in
  A$. Using the assumption that $s^+(x^a) = s^+(x)^a$, we see that
  $\sigma^+(\mu^{\pm}(x^a,y^a)) = \sigma^+(\mu^{\pm}(x,y))^a$ for all
  $x,y \in V$ and $a \in A$. Let $V^\pm(y) = \{ x \in V \mid x \in
  \sigma^+(\mu^\pm(x,y)) \}$. We have shown that $V^\pm(y^a) =
  V^\pm(y)^a$, and thus $V^+(y) = V^-(y) \iff V^+(y^a) = V^-(y^a)$
  for all $y \in V$. The lemma follows since $V \ni x \to \sigma^+(x)$
  is an orientation for $(\mu^-,\mu^+)$ if and only if $V^+(y) =
  V^-(y)$ for all $y \in V$.
\end{proof}

 Let $\mu$ be a Steiner quasigroup on $V$. We shall now discuss
 (orientable) transversals to $\mu$
 (Definition~\ref{defn:steinertria}). 

  The shift map of $\mu$, 
  \begin{equation}
    \label{eq:sigmaT}
    \sigma \colon V \times \Sigma(V) \to \Sigma(V), \quad y^{\sigma(x,T)}
    = \mu(x,\mu^T(x,y)), \qquad x,y\in V, \quad T \in \Sigma(V),
  \end{equation}
  takes $(x,T)$ to the transition permutation $\sigma(\mu,x,\mu^T)$
  from $\mu^T$ to $\mu$. Thus $\mu^T(x,y)^{\sigma(x,T)} = \mu(x,y)$.
  We call $\sigma(x,T)$ the $x$-shift of $T$ and $\sigma$.

\begin{prop}\label{prop:sigmaT}
  Let $\mu$ be a Steiner quasigroup on $V$ and $T \in \Sigma(V)$ a
  symmetry of $V$.
  \begin{equation*}
    \text{$T$ is transversal to $\mu$} \iff
    \forall x \in V \colon 
    \text{$\sigma(x,T)$ has cycle structure $1^1m^2$.}
  \end{equation*}
  If $T$ is transversal to $\mu$ then $T$ is orientably transversal to
  $\mu$ if and only if it there are transitive size $m$ orbits
  $\sigma^+(x,T)$ for $\sigma(x,T)$ at every $x \in V$ such that
 \begin{equation*}
  x^T \in \sigma^+(\mu(x^T,y^T),T) \iff x^T \in \sigma^+(\mu(x,y)^T,T) 
\end{equation*}
  for all $x,y \in V$.
\end{prop}

The auto\m\ group of $\mu^T$ is $\Sigma(\mu^T)=\Sigma(\mu)^T$.  The
intersection 
\begin{equation}
  \label{eq:smucapsmuT}
\Sigma(\mu) \cap \Sigma(\mu)^T = \{ a \in \Sigma(V) \mid
 \mu^a = \mu, \, \mu^{Ta} = \mu^T \}  
\end{equation}
acts on the simplicial complex $S(\mu) \cup S(\mu)^T$. If $T$ is
transversal to $\mu$ then $\Sigma(\mu) \cap \Sigma(\mu)^T$ is
contained in the auto\m\ group of the triangulated surface $S(\mu)
\cup S(\mu)^T$ with index at most $2$.

\pagebreak

\begin{prop}\label{prop:transversal}
  Let $T \in \Sigma(V)$ be a symmetry of $V$.
  \begin{enumerate}

  \item \label{prop:transversal1}   
  If $a \in \Sigma(V)$ is a symmetry of $V$ then
  \begin{equation*}
  a \in \Sigma(\mu) \iff \mu^a= \mu \iff \mu^{aT}=\mu^T 
  \iff \forall
  x \in V \colon \sigma(x,aT) = \sigma(x,T) \iff \forall x \in V
  \colon \sigma(x^a,T^a) = \sigma(x,T)^a.
  \end{equation*}

  \item \label{prop:transversal3} 
    If $a \in \Sigma(\mu)$ is a symmetry of $\mu$ then 
    \begin{equation*}
      a \in \Sigma(\mu) \cap \Sigma(\mu)^T \iff 
    \mu^{Ta} = \mu^{T} \iff 
    \forall x \in V \colon \sigma(x,Ta) = \sigma(x,T) \iff
    \forall x \in V  \colon \sigma(x^a,T) = \sigma(x,T)^a.
    \end{equation*}

  \end{enumerate} 
\end{prop}
\begin{proof}
  \noindent \eqref{prop:transversal1} The first part is clear since
  $\Sigma(\mu)$ is the isotropy subgroup at $\mu$ under the right
  action \eqref{eq:STSxSigmaV} of $\Sigma(V)$ on $\mathrm{STS}(V)$.
  The condition $\sigma(x^a,T^a) = \sigma(x,T)^a$ for all $x \in V$ is
  equivalent to $\mu^{T^a}(x^a,y^a) = \mu^T(x,y)^a$ or
  $\mu^a(x^{T^{-1}},y^{T^{-1}}) = \mu(x^{T^{-1}},y^{T^{-1}})$ for all
  $x,y \in V$.

  \noindent \eqref{prop:transversal3} 
  Clearly, $a \in \Sigma(\mu^T) \iff \mu^{Ta}
  = \mu^T$, as $\Sigma(\mu^T)$ is the isotropy subgroup at $\mu^T$. 
  Next
  \begin{equation*}
    \mu^{Ta} = \mu^T \iff 
    \forall x,y \in V \colon \mu(x,\mu^{Ta}(x,y)) = \mu(x,\mu^T(x,y))
    \iff \\
    \forall x \in V \colon \sigma(x,T) = \sigma(x,Ta)
  \end{equation*}
  establishes the second equivalence. Since $a \in \Sigma(\mu)$ we
  have, as in \eqref{prop:transversal1}, that $\mu^T(x,y)^a =
  \mu^{Ta}(x^a,y^a)$ for all $x,y \in V$. This leads to a new chain a
  equivalences,
  \begin{multline*}
    \forall x \in V \colon \sigma(x^a,T) = \sigma(x,T)^a \iff
    \forall x,y \in V \colon 
    \mu(x^a,\mu^T(x^a,y^a)) = \mu(x,\mu^T(x,y))^a \iff \\
    \forall x,y \in V \colon 
    \mu(x^a,\mu^T(x^a,y^a)) = \mu(x^a,\mu^{Ta}(x^a,y^a)) \iff 
    \mu^{Ta} = \mu^T,
  \end{multline*}
   proving the final equivalence.
\end{proof}

\begin{rmk}
  Suppose that $a \in \Sigma(\mu)$, $T \in \Sigma(V)$ is transversal
  to $\mu$, and that $\sigma(x_0^a,T) = \sigma(x_0,T)^a$ for some $x_0
  \in V$. Let $y_0 \neq x_0$ be a point of $V$ distinct from $x_0$.
  For any $j \geq 0$,
  \begin{equation*}
     y_0^{\sigma(x_0,T)^j} \xrightarrow{a}  
     Y_0^{\sigma(X_0,T)^j}, \qquad
     \mu(x_0,y_0)^{\sigma(x_0,T)^j} \xrightarrow{a}
     \mu(X_0,Y_0)^{\sigma(X_0,T)^j}.
  \end{equation*}
  Thus $a$ is completely determined by its values, $X_0=x_0^a$ and
  $Y_0=y_0^a$, on the two points $x_0$ and $y_0$. (The points $y_0$
  and $\mu(x_0,y_0)$ are not in the same orbit of $\sigma(x_0,T)$.)
  Therefore the order of $\Sigma(\mu) \cap \Sigma(\mu)^T$ is at most
  $n(n-1)$.
\end{rmk}

   Let $a \in \Sigma(\mu)$ be a symmetry of
   $\mu$. Proposition~\ref{prop:transversal}.\eqref{prop:transversal1}
   shows that 
     \begin{multline*}
     \text{$T$ is (orientably) transversal to $\mu$} \iff
     \text{$aT$ is  (orientably) transversal to $\mu$}  \iff \\
     \text{$T^a$ is  (orientably) transversal to $\mu$}  \iff
     \text{$Ta$ is  (orientably) transversal to $\mu$}.
      \end{multline*}
      Note that if $\sigma^+(x)$ is an orientation for $(\mu,\mu^T)$
      then $\sigma^+(x^{a^{-1}})^a$ is an orientation of
      $(\mu,\mu^{T^a})$.  Thus transversals to the Steiner quasigroup
      $\mu$ are really elements of $\Sigma(\mu) \backslash \Sigma(V) /
      \Sigma(\mu)$.

\begin{defn}\label{defn:transmu}
  For any Steiner quasigroup $\mu$ on $V$,
  \begin{equation*}
    \mathrm{T}^{(+)}(\mu) = 
    \{ \Sigma(\mu)T\Sigma(\mu) \in 
    \Sigma(\mu) \backslash \Sigma(V) /\Sigma(\mu) \mid
    \text{$T$ is (orientably) transversal to $\mu$} \} 
  \end{equation*}
  denotes the set of equivalence classes of (orientable) transversals
  to $\mu$. For any subgroup $A \leq \Sigma(\mu)$,
  \begin{equation*}
     \mathrm{T}^{(+)}(\mu)_{\gtrsim A} = \{ T \in T^{(+)}(\mu) \mid
     \Sigma(\mu) \cap \Sigma(\mu)^T \gtrsim A \}
  \end{equation*}
  is the set of (orientable) transversals $T$ for which $\Sigma(\mu)
  \cap \Sigma(\mu)^T$ is superconjugate in $\Sigma(\mu)$ to $A$.
\end{defn}

      The next Proposition~\ref{prop:orbitcat} deals with sets of
      double cosets of the form $H \backslash G \slash H$.  The
      following notation will be used: When $G$ is a group and $A \leq
      H \leq G$ and $K \leq G$ are subgroups of $G$ then
\begin{itemize}
\item $K \gtrsim_G H$ means that $K$ is superconjugate in $G$ to
  $H$ (a $G$-conjugate of $K$ is a supergroup of $H$),
\item $N_G(H,K) =
\{ g \in G \mid H^g \leq K\}$ is the transporter set,
\item $A^{H} = \{A^h \mid h \in H\}$ is the set of
  $H$-conjugates of $A$,
\item $A^G_{\leq H}$ is the set of $G$-conjugates of $A$ contained in $H$.
\end{itemize}

\begin{prop} \label{prop:orbitcat}
  Let $H$ be a subgroup of $G$ and $A$ a subgroup of $H$.
  \begin{enumerate}
  \item \label{prop:orbitcat1} The map
    \begin{equation*}
      N_G(A,H) \xrightarrow{g \to Hg^{-1}H} 
      \{ g \in H \backslash G \slash H \mid H \cap H^g \gtrsim A \}
    \end{equation*}
    is surjective.
  \item \label{prop:orbitcat3} Let $g_1,\ldots,g_t \in N_G(A)
    \backslash G \slash H$ be distinct elements such that $A^G_{\leq H}
    = \{A^{g_1},\ldots,A^{g_t} \}$. Then
    \begin{equation*}
      N_G(A,H) = \coprod_{1 \leq j \leq t} g_jN_G(A^{g_j})H.
    \end{equation*}
  \item \label{prop:orbitcat4} If $A^G_{\leq H} \subseteq A^H$ (i.e.,
    every $G$-conjugate of $A$ in $H$ is $H$-conjugate to $A$) then
    \begin{equation*}
      N_H(A) \backslash N_G(A) \slash N_H(A) 
      \xrightarrow{N_H(A)gN_H(A) \to HgH} 
       \{ g \in H \backslash G \slash H \mid H \cap H^g \gtrsim A \}
    \end{equation*}
    is surjective and it is bijective if also $N_G(A) \leq N_G(H)$.
  \end{enumerate}
\end{prop}
\begin{proof}
  \noindent \eqref{prop:orbitcat1} If $g \in N_G(A,H)$ then $A^g \leq H$
  or $H^{g^{-1}} \geq H$. Thus the map $g \to g^{-1}$ takes the set on
  the left to the set on the right. This map is surjective because
  \begin{equation*}
    H^{g} \gtrsim A \iff
    \exists h \in H \colon H^{g} \geq A^{h} \iff
    \exists h \in H \colon A^{hg^{-1}} \leq H \iff
    \exists h \in H \colon hg^{-1} \in N_G(A,H)
  \end{equation*}
  and the image of $hg^{-1}$ in $H \backslash G \slash H$ is $HgH$.

  \noindent \eqref{prop:orbitcat3} Let $g \in N_G(A,H)$. Since $A^g
  \leq H$ is conjugate in $H$ to $A^{g_j}$ for some $j$, $A^{g_j} =
  A^{gh}$ for some $h \in H$. Then $ghg_j^{-1}$ normalizes $A$, $gh
  \in N_G(A)g_j$ and $g = (gh)h^{-1} \in N_G(A)g_jH =
  g_jN_G(A^{g_j})H$. 

  \noindent \eqref{prop:orbitcat4} Since $A^G_{\leq H} = A^H$,
  $N_G(A,H) = N_G(A)H$ according to \eqref{prop:orbitcat3} and the map
  \begin{equation*}
    N_G(A) \xrightarrow{g \to HgH} 
     \{ g \in H \backslash G \slash H \mid H \cap H^g \gtrsim A \}
  \end{equation*}
  is surjective by \eqref{prop:orbitcat1}. Suppose that $g_1,g_2 \in
  N_G(A)$ have identical images in $H \backslash G \slash H$. Then
  $h_1g_1 = g_2h_2$ for some $h_1,h_2 \in H$. Thus $g_2^{-1}g_1 =
  g_2^{-1}h_1^{-1}g_2h_2 = (h_1^{-1})^{g_2}h_2 \in H$ since $H^{g_2} =
  H$ as all elements of $G$ normalizing $A$ normalize $H$ by
  assumption. Thus $g_2^{-1}g_1 \in N_H(A)$ and $g_1N_H(A) =
  g_2N_H(A)$. 
\end{proof}

We conclude from
Proposition~\ref{prop:orbitcat}.\eqref{prop:orbitcat4} that when $A
\leq \Sigma(\mu)$ has the property that any $\Sigma(V)$-conjugate of
$A$ in $\Sigma(\mu)$ is $\Sigma(\mu)$-conjugate to $A$ (i.e., if $A$ is a
Sylow $p$-subgroup of $\Sigma(\mu)$ for some prime $p$) then there is
a surjection
 \begin{equation}
   \label{eq:surjection}
   N_{\Sigma(\mu)}(A) \backslash
   \{ T \in N_{\Sigma(V)}(A) \mid 
   \text{$T$ is transversal to $\mu$} \} \slash N_{\Sigma(\mu)}(A)
   \twoheadrightarrow \mathrm{T}(\mu)_{\gtrsim A}
 \end{equation}
 that is bijective if also $N_{\Sigma(V)}(A) \leq
 N_{\Sigma(V)}(\Sigma(\mu))$.

Let $S \in \mathrm{STS}(V)$ be  a Steiner triple system on $V$ and $T
 \in \Sigma(V)$ a transversal to $S$. The $1$-chromatic number of the
 Steiner surface $S \cup S^T$ is $\chs 1{S \cup S^T} = |V| = n$ as all
 Steiner surface are neighborly. But what about the $2$-chromatic
 number? We discuss this question for two families of Steiner triple
 systems, the Bose Steiner triple systems and the projective Steiner
 triple systems.

 \section{The Bose Steiner triple systems --- an infinite series of neighborly surfaces with $\chi_2=3$}
\label{sec:bose}

We show in this section that Bose Steiner triple systems generate an infinite sequence 
of neighborly triangulated surfaces with unbounded $1$-chromatic numbers,
but with constant $2$-chromatic number $\chi_2=3$.

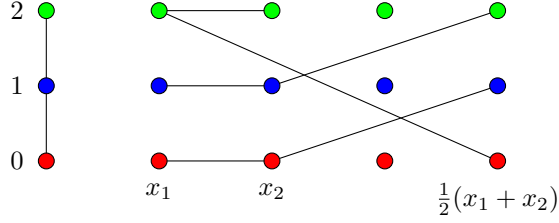
\begin{figure}[t]
  \centering
  
\begin{tikzpicture}[vertex/.style= {shape=circle,  
   fill={#1!100}, minimum size =
  6pt, inner sep =0pt,draw}, vertex/.default=black, xscale=1.5]

  \draw (1,1) --  (1,2) -- (1,3);
  \draw (2,1) -- (3,1) -- (5,2);
  \draw (2,2) -- (3,2) -- (5,3);
  \draw (5,1) -- (2,3) -- (3,3) ;
  
  \foreach \x in {1,...,5} {
    \foreach \y/\c in {1/red,2/blue,3/green} {
      \node[vertex=\c] at (\x,\y) {};};};
  
   \node[below=5pt] at (2,1) {$x_1$};
   \node[below=5pt] at (3,1) {$x_2$};
   \node[below=5pt] at (5,1) {$\frac{1}{2}(x_1+x_2)$};

   \node[left=5pt] at (1,1) {$0$};
   \node[left=5pt] at (1,2) {$1$};
   \node[left=5pt] at (1,3) {$2$};

\end{tikzpicture}
  \caption{The Bose Steiner triple systems $B(s)$ are $(3,2)$-colorable.}
  \label{fig:bosecol}
\end{figure}

Let $s>1$. In the ring $\Z/(2s+1)\Z$ of integers modulo the odd integer $2s+1$, 
$2$ is invertible with $2^{-1} = s+1$.  Let $V=\Z/(2s+1)\Z \times \Z/3\Z$.  
Then $V$ has order $n= |V| = 3(2s+1)$ and $n \equiv 3 \bmod 6$. Let also $m=3s+1$ so that $n=2m+1$.
The {\em Bose Steiner triple system\/} on $V$ \cite{bose39},
  \begin{eqnarray*}
    B(s) & = & \{ \{x\} \times \Z/3\Z \mid x \in \Z/(2s+1)\Z  \} \\
            &   & \cup\,  \{ \{(x_1,y),(x_2,y), ((x_1+x_2)/2,y+1)\} \mid x_1,x_2 \in  \Z/(2s+1)\Z,  y \in \Z/3\Z, x_1 \neq x_2 \},
  \end{eqnarray*}
  consists of $2s+1$ `vertical' and $3 \binom{2s+1}2 = 3s(2s+1)$ `slanted'
  triangles.
  Because
  \begin{equation*}
    n \bmod 12  =
    \begin{cases}
      3 \bmod 12, & \text{$s$ even}, \\
      9 \bmod 12, & \text{$s$ odd} 
    \end{cases}
  \end{equation*}
  Bose systems $B(s)$ with odd $s$ admit nonorientable
  transversals only, but for even $s$ orientable transversals may exist.

\enlargethispage*{1.5mm}  
  
  The following proposition and theorem, nearly all of which was proved in
  \cite{grannell1998}, contain an example of an \emph{orientable}
  transversal to $B(s)$. The proof presented here does not use results
  from topological graph theory so may represent a partial response to
  \cite[Problem 1]{grannell1998}.

  \begin{prop}\label{prop:bose}\cite[Theorem 1]{grannell1998}.
    Assume that $s \geq 2$ is even. The permutation
 \begin{equation*}
    (x,y)^{T} =
    \begin{cases}
      (x,0), & y=0, \\ (x+1,2), & y =1, \\ (x+1,1), & y =2
    \end{cases}
  \end{equation*}
  is orientably transversal to $B(s)$. 
  \end{prop}  

  \begin{proof}
    Here, we give an alternative proof to \cite{grannell1998}.
    By Proposition~\ref{prop:sigmaT} we must determine the cycle
    structure of the shifts $\sigma(u,T)$ defined by equation
    \eqref{eq:sigmaT} for all $u \in V$. However, $\Sigma(\mu) \cap
    \Sigma(\mu)^T$ contains the subgroup $\Z/(2s+1)\Z$ so by
    Lemma~\ref{lemma:shiftA} it suffices to do this for the three $u$
    in $V_{\Z/(2s+1)\Z} = \{0\} \times \Z/3\Z$.

  The formulas 
  \begin{equation*}
    \mu((0,0),(x,y)) =
    \begin{cases}
      (\frac{1}{2}x,1), & y=0, x \neq 0, \\
      (2x,0), & y=1, x \neq 0, \\
      (-x,2), & y=2, x \neq 0,
    \end{cases} \qquad
     \mu^T((0,0),(x,y)) =
     \begin{cases}
       (\frac{1}{2}x-1,2), & y=0, x \neq 0, \\
       (-x+2,1), & y=1, x \neq 1, \\
       (2x-2,0), & y=2, x \neq 1
     \end{cases}
  \end{equation*}
 imply that
 \begin{equation*}
    (x,y)^{\sigma((0,0),T)} =
    \mu((0,0),\mu^T((0,0),(x,y))) =
    \begin{cases}
      (-\frac{1}{2}x-1,2), & y=0, x \neq 0, x \neq -2, \\
      (-2x+4,0), & y=1, x \neq 1, x \neq 2, \\
      (x-1,1), & y=2, x \neq 1.
    \end{cases}
 \end{equation*}
 This expression can be used to show that the shift $\sigma((0,0),T)$
 has cycle structure $1^1m^2$. The orbits through $(0,1)$ and
 $\mu((0,0),(0,1)) = (0,2)$,
 \begin{align*}
      &
   \overbrace{(0,1) \to \cdots \to (1,1)}^{3s/2} \to 
   \overbrace{(2s,2) \to \cdots \to (4,2)}^{3s/2-3} \to (3,1) \to
   (-2,0)  \to (0,1), \\
   &\underbrace{(0,2) \to  \cdots \to (1,2)}_{3s/2} 
   \to 
   \underbrace{(2,0) \to \cdots \to
   (5,0)}_{3s/2-3} \to (3,2) \to (2,1) \to (0,2),
 \end{align*}
 are disjoint and both have length $m$. It may be helpful to
 observe that $T^3(x,0)=(x+8,0)$, $T^3(x,1)=(x-4,1)$,
 $T^3(x,2)=(x-4,2)$ provided $T^3$ does not involve any of the
 exceptions $(0,0),(-2,0),(1,1),(2,1),(1,2)$.  

 Completely analogous arguments show that also $\sigma((0,1),T)$ and
 $\sigma((0,2),T)$ have cycle structure $1^1m^2$. We have now shown
 that $T$ is a transversal to $B(s)$.
 \end{proof}

\begin{thm}\label{thm:orientablebosesteiner}
 The orientable combinatorial Steiner
  surface $B(s) \cup B(s)^T$ has genus $\frac{1}{2}s(6s-1)$ and
  chromatic numbers $\chs 1{B(s) \cup B(s)^T} = 3(2s+1)$, $\chs 2{B(s)
    \cup B(s)^T} = 3$. The horizontal shift $h(x,y)=(x+1,y)$ is an
  orientation preserving auto\m\ of $B(s) \cup B(s)^T$.
  \end{thm}

  \begin{proof}
 If we let $v(u)=u+(0,1)$ denote the vertical and $h(u)=u+(1,0)$ the
 horizontal shift of $V$, then $V \ni u \to v(u)^{\gen{\sigma(u,T)}}$
 is an orientation for $B(s) \cup B(s)^T$ (Lemma~\ref{lemma:LAmu}) and
 $h \in \Sigma(B(s)) \cap \Sigma(B(s))^T$ is an orientation preserving
 auto\m\ of $B(s) \cup B(s)^T$.  

 The projection $\Z/(2s+1)\Z \times \Z/3\Z \to \Z/3\Z$ is a
 $(3,2)$-coloring of $B(s) \cup B(s)^T$ which has no horizontal
 triangles (Figure~\ref{fig:bosecol}). 
 On the other hand, all Steiner
 triple system with more than one triple have $2$-chromatic number at
 least $3$ \cite{rosa70}.
  \end{proof}

\begin{rmk}
Figure~\ref{fig:bosecol} illustrates the case $s=2$. Modulo $2s+1$ we have $(1+2)/2 \equiv 4$. 
The first coordinates of triangles in the Bose system are either constant, as in the vertical triangle on the left, 
or, if they are non-vertical, they have first coordinates $(x_1,x_2,(x_1+x_2)/2)$.
The figure shows the three nonvertical triangles with first coordinates $x_1=1$, $x_2=2$, and hence the third first
coordinate is $4$ (never $3$).
\end{rmk}

  For the sake of completeness, we quote the following result from
  \cite{soloveva2007} exhibiting a {\em nonorientable\/} transversal
  to $B(s)$ for any $s \geq 1$. 

\begin{prop}\cite[Theorem 3]{soloveva2007}
  Assume that $s \geq 1$. The permutation
  \begin{equation*}
    (x,y)^T =
    \begin{cases}
      (x,0), & y=0, \\
      (x+1,1), & y=1, \\
      (x+2,2), & y=2
    \end{cases}
  \end{equation*}
  is nonorientably transversal to $B(s)$. 
\end{prop}

  One may use Lemma~\ref{lemma:propagate} to show that $B(s) \cup B(s)^T$ 
  is a nonorientable Steiner surface. The statements about its chromatic numbers 
  are proved in the same way as in the proof of Theorem~\ref{thm:orientablebosesteiner}.
   
\begin{thm}
  The nonorientable combinatorial Steiner surface $B(s) \cup B(s)^T$ has genus
  $s(6s-1)$ and chromatic numbers $\chs 1{B(s) \cup B(s)^T}=3(2s+1)$,
  $\chs 2{B(s) \cup B(s)^T}=3$.
\end{thm}

  The auto\m\ group of $B(s)$ contains $(\Z/(2s+1)\Z \times \Z/3\Z) \rtimes
  (\Z/(2s+1)\Z)^\times$ and, in fact, equals this group except for
  $s=1$ or $s=4$ \cite[Theorem 3.1]{lovegrove:2003}.

\section{The projective Steiner triple systems --- a non-orientable surface with $\chi_2\in\{5,6\}$}
\label{sec:PGd2}

Let $d>2$, $N=2^d$ and $m=2^{d-1}-1$ so that $N-1=2m+1$.  Let
$\F_{N}$ be the field with $N$ elements and $\F_N^\times$ the set of
nonzero elements in $\F_N$.  The \emph{projective geometry Steiner triple system}
is the set
  \begin{equation*}
  \mathrm{PG}(N) = \{ \{a,b,c\}  
  \mid a,b,c \in \F_N^\times,\quad a \neq b, \quad a+b+c=0 \}    
  \end{equation*}
  of $3$-subsets of the $(N-1)$-set $\F_N^\times$.  The corresponding
  Steiner quasigroup is $\mu(a,b)=a+b$, $a \neq b$.  The auto\m\ group
  of $\mathrm{PG}(N)$ is the linear group $\Sigma(\mathrm{PG}(N)) =
  \GL d2$. The Singer cycle (multiplication by a primitive element of
  $\F_N$) is an auto\m\ of $\mathrm{PG}(N)$ acting transitively on the
  vertex set.  Because $2^4 \equiv 2^2 \bmod 12$ and $2^5 \equiv 2^3
  \bmod 12$,
  \begin{equation*}
     2^d  - 1  \bmod 12 = 
  \begin{cases}
    7 \bmod 12, & \text{$d$  odd,} \\
    3 \bmod 12, & \text{$d$  even,} \\
  \end{cases}
  \end{equation*}
  so that all projective Steiner triple system potentially allow
  orientable transversals.

  Let $\Sigma_0(\F_N)$ be the group of all permutations of $\F_N$
  fixing $0$. Including also $0$ in the domain, the shift-map
  \eqref{eq:sigmaT} becomes
  \begin{equation*}
    \sigma \colon \F_N \times \Sigma_0(\F_N) \to \Sigma_0(\F_N)
  \end{equation*}
  given by $y^{ \sigma(x,T)} = (y^{-T}+x^{-T})^T+x$ for all $x,y \in
  \F_N$ and $T \in \Sigma_0(\F_N)$.  
  These shifts have the following properties 
   \begin{itemize}
   \item $\sigma(x,\mathrm{Id}) = \mathrm{Id}$ for all $x \in \F_N$,
   \item $\sigma(0,T)=\mathrm{Id}$ for all $T \in \Sigma_0(\F_N)$,
   \item $\sigma(x,aT) = \sigma(x,T)$ and $\sigma(x^a,T^a) =
     \sigma(x,T)^a$ for all $a \in \GL d2$
     (Proposition~\ref{prop:transversal}.\eqref{prop:transversal1}),
   \item $\sigma(x,T)$ fixes $0$ and $x$ for all $T \in
     \Sigma_0(\F_N)$,
   \end{itemize}
   where $\mathrm{Id}$ is the identity permutation.

   Fix an element $z \in \F_N$ distinct from $0$ and $1$.  A
   permutation $T \in \Sigma_0(\F_N)$ is by
   Proposition~\ref{prop:sigmaT} transversal to $\mathrm{PG}(N)$ if
   and only if:
  \begin{enumerate}
  \item For all $x \in \F_N^\times$, the $x$-shift $\sigma(x,T)$
    partitions $\F_N - \{0,x\}$ into two equally sized orbits.
  \end{enumerate}
  It is orientably transversal to $\mathrm{PG}(N)$ if also
\begin{enumerate}
  \setcounter{enumi}{1}
 \item $\forall u_1,u_2 \in z^{\gen{\sigma(1,T)}} \forall x \in
    1^{\gen{\sigma(u_1,T)}} \cap 
    1^{\gen{\sigma(u_2,T)}} \colon u_1^{\gen{\sigma(x,T)}} = 
    u_2^{\gen{\sigma(x,T)}}$ according to Lemma~\ref{lemma:propagate}.
\end{enumerate}

The next lemma shows that transversals to $\mathrm{PG}(N)$ are
permutations of $\F_N$ that are highly non-linear.

\begin{lemma}\label{lemma:transPGN}
  Let $T \in \Sigma_0(\F_N)$ be a transversal to $\mathrm{PG}(n)$.
  \begin{enumerate}
  \item $T$ is not linear. \label{lemma:transPGN1}
  \item The image, $U^T$, of any linear subspace $U$ of $\F_N^+ =
    \F_2^d$ is not a linear subspace when $1 < \dim_{\F_2} U < d$.
    \label{lemma:transPGN2}
  \end{enumerate}
\end{lemma}
\begin{proof}
  Proposition~\ref{prop:transversal} shows that no transversal to
  $\mathrm{PG}(N)$ can be linear since the identity is not
  transversal.
 
  Let $U$ be a proper linear subspace of dimension $>1$ of the
  $d$-dimensional $\F_2$-vector space $\F_N^+$ underlying the field
  $\F_N$. Suppose that the image, $U^T$, of $U$ is again a linear
  subspace. Choose two distinct and nonzero elements $x,y \in U^T$.
  This is possible as $\dim_{\F_2} U^T > 1$.  Note that the orbit of
  the $x$-shift $\sigma(x,T)$ through $y$ stays inside $U^T$ and does
  not contain $0$ and $x$.  The size of the orbit $y^{\sigma(x,T)}$ is
  at most $|U|-2 \leq 2^{d-1}-2 < 2^{d-1}-1 = m$. Thus $\sigma(x,T)$
  cannot have cycle structure $1^2m^2$.
\end{proof}

It is convenient to represent permutations of $\F_N$ by polynomials
with coefficients in $\F_N$.

\begin{rmk}[Permutation polynomials and (orientable) surface polynomials]
  For any permutation $T \in \Sigma_0(\F_N)$ there is a unique
  permutation polynomial \cite{lidl88} $P_T \in \F_N[X]$ of degree
  $N-2$ such that $P_T(x)=x^T$ for all $x \in \F_N$. The coefficients
  of $P_T = \sum_{i=1}^{N-2}a_iX^i$ are given by
\begin{equation*}
     (a_1,\ldots,a_{N-2})(x^{ij})_{1 \leq i,j \leq N-2} = 
     ((x^j)^T)_{1 \leq j \leq N-2},
   \end{equation*}
   where $x$ is  primitive in $\F_N$.
   If $T \in \Sigma_0(\F_N)$ is an (orientable) transversal to
   $\mathrm{PG}(N)$ we say that $P_T$ is an (orientable) surface
   polynomial. With this terminology, for a polynomial $P \in
   \F_{N}[X]$ we have that
\begin{equation*}
  \text{$P$ is an (orientable) surface polynomial} \iff
  \text{$(\mathrm{PG}(N),\mathrm{PG}(N)^P)$ is an (orientable)  
   Steiner surface},
\end{equation*}
where the Steiner surface has genus $\frac{1}{12}(N-5)(N-4)$
if orientable and the double of this number if nonorientable.
\end{rmk}

\begin{exmp}[Surface monomials]\label{exmp:monomial}
  Suppose that $\mathrm{GCD}(N-1,r)=1$ so that $T = X^r$ is a
  permutation monomial with $0^T=0$ and $1^T=1$.  Since $\mu^T(x,y)
  = (x^s+y^s)^r$, where $rs \equiv 1 \bmod N-1$ by Lemma~\ref{lemma:shiftA}, we see that
  $\F_{N}^\times \leq \Sigma(\mu) \cap \Sigma(\mu)^T$ and hence  by Lemma~\ref{lemma:shiftA}: 
  \begin{equation*}
    \text{$X^r$ is a surface monomial.} \iff
    \text{The permutation $\sigma(1,X^r) \colon y \to 1+(1+y^s)^r$ of $\F_{N}$ 
      has cycle structure $1^2m^2$.}
  \end{equation*}\\
    Lemma~\ref{lemma:transPGN} shows that
  $X^r$ is {\em not\/} a surface monomial if
  \begin{itemize}
  \item $d=d_1d_2$ with $d_1,d_2>1$
  so that $X^r$ normalizes the subfield $\F_{2^{d_1}}$ of dimension $1 <
  2^{d_1} < 2^d$, or,
 \item  $r$ is a power of $2$, so that $X^r$ is linear.
  \end{itemize}
  Thus surface monomials only occur when $d$ is prime.  See
  \cite[Theorem 5]{rifa2014} for a list of surface monomials for
  small values of $N$.  (Surface binomials are much more mysterious.)
\end{exmp}

\begin{prop}[$d=3$]\label{prop:PG8}
  $|\mathrm{T}(\mathrm{PG}(8))| = 1 = |\mathrm{T}^+(\mathrm{PG}(8))|$.
  If $T$ is the unique transversal to $\mathrm{PG}(8)$ then $\GL 32
  \cap \GL 32^T = N_{\GL 32}(\F_{2^3}^\times)$ is the normalizer of
  the Sylow $7$-subgroup of $\GL 32$.
\end{prop}
\begin{proof}
  The quotient set $\GL 32 \backslash \Sigma_0(\F_{2^3}) \slash \GL 32$
  has four elements. Using Proposition~\ref{prop:sigmaT} and
  Lemma~\ref{lemma:propagate} we see that only one of the four double
  cosets is transversal to $\mathrm{PG}(8)$ and that this transversal
  is orientable. 
\end{proof}

$P_3=X^3$ is one of the orientable transversal to $\mathrm{PG}(8)$.
Let $a$ be a primitive element of $\F_8$. Then $s^+(x) =ax$ defines an
orientation (Lemma~\ref{lemma:LAmu}) of the Steiner surface
$\mathrm{PG}(8) \cup \mathrm{PG}(8)^{P_3}$
(Figure~\ref{fig:stinertria}) and $N_{\GL 23}(\F_{2^3}^\times) =
\F_{2^3}^\times \rtimes \mathrm{Gal}(\F_{2^3},\F_2)$ acts
by orientation preserving auto\m s.

\begin{prop}[$d=4$]\label{prop:PG16}
  $|\mathrm{T}(\mathrm{PG}(16))|=4$ and
  $|\mathrm{T}^+(\mathrm{PG}(16))|=1$.  If $T$ is a nonorientable
  transversal then $\GL 42 \cap \GL 42^T$ is trivial in two cases and
  is the Sylow $5$-subgroup of $\GL 42$ in one case.  If $T$ is an
  orientable transversal then $\GL 42 \cap \GL 42^T$ is the Sylow
  $5$-subgroup of $\GL 42$.
\end{prop}
\begin{proof}  
  It is possible with a computer to go through all $ T \in \GL 42
  \backslash \Sigma_0(\F_{2^4}) \slash \GL 42$ and use
  Proposition~\ref{prop:sigmaT} to check if $T$ is an orientable
  transversal. This gives an explicit description of the set
  $\mathrm{T}(\mathrm{PG}(16))$ of transversals.
\end{proof}

$P_4=aX^{11}+X^6+X$, where $a$ is any primitive elements of $\F_{16}$
such that $a^4+a+1=0$, is an orientable transversal to
$\mathrm{PG}(16)$ and $\GL 42 \cap \GL42^{P_4} = \gen{a^3}$ is the
Sylow $5$-subgroup of $\F_{2^4}^\times$ and $\GL 42$.  The primitive
element $a$ can be chosen such that $s^+(x)=a^5x$ defines an orientation
(Lemma~\ref{lemma:LAmu}) on the orientable Steiner surface
$\mathrm{PG}(16) \cup \mathrm{PG}(16)^{P_4}$ such that $\GL 42 \cap
\GL 42^{P_4}$ acts by orientation preserving auto\m s.  There are no
surface monomials as $d$ is a composite number
(Example~\ref{exmp:monomial}).

\begin{prop}[$d=5$, \mbox{\cite[Example 1]{rifa2014}}]\label{prop:PG32}
  Let
  $p$ is a prime divisor of $|\GL 52|$ and $S_p$ the Sylow
  $p$-subgroup of $\GL 52$. Then
  $|\mathrm{T}^+(\mathrm{PG}(32))_{\gtrsim S_p}| = 0$ and
  \begin{equation*}
    |\mathrm{T}(\mathrm{PG}(32))_{\gtrsim S_p}| =
    \begin{cases}
      2, & p=31, \\   65, & p=5, \\ 0, & p=2,3,7.
    \end{cases} 
  \end{equation*}
  For any $T \in \mathrm{T}(\mathrm{PG}(32))_{\gtrsim S_{31}}$, $\GL
  52 \cap \GL 52^T = N_{\GL 52}(\F_{32}^\times) =\F_{32}^\times
  \rtimes \mathrm{Gal}(\F_{2^5},\F_2)$. For any $T \in
  \mathrm{T}(\mathrm{PG}(32))_{\gtrsim S_{5}}$, 
  $\GL 52 \cap \GL 52^T$ is $\mathrm{Gal}(\F_{2^5},\F_2)$ (in $63$
  cases) or $N_{\GL 52}(\F_{32}^\times)$ (in $2$ cases).
\end{prop}
\begin{proof} 
  When $d=5$ it is no longer feasible to search for transversals
  in the complete
  set $\GL d2 \backslash \Sigma_0(\F_{2^d}) / \GL d2$ of double cosets
  as we did for $d=4$.  However, by \eqref{eq:surjection} all
  transversals in $\mathrm{T}(\mu)_{\gtrsim S_p}$ are represented by
  double cosets of the smaller set $N_{\GL 52}(S_p) \backslash
  N_{\Sigma_0(\F_{2^5})}(S_p) / N_{\GL 52}(S_p)$, and it is possible
  with a computer to locate the coset representatives of (orientable)
  transversals using the tests of Proposition~\ref{prop:sigmaT}.
  We find that the set
  $\mathrm{T}(\mathrm{PG}(32))_{\gtrsim S_p}$ is empty except for
  $p=31$ and $p=5$.
\end{proof}

$P_5=X^5$ is a nonorientable transversal to $\mathrm{PG}(32)$. No 
  orientable transversals to $\mathrm{PG}(32)$ are known.

\begin{prop}[$d=7$, \mbox{\cite[Example 2]{rifa2014}}]\label{prop:PG128}
  $|\mathrm{T}(\mathrm{PG}(128))_{\gtrsim \F_{2^8}^\times}| = 8$. For
  any $T \in \mathrm{T}(\mathrm{PG}(128))_{\gtrsim \F_{2^8}^\times}$,
  $T$ is non\-orientable and the group $\GL 72 \cap \GL 72^T = N_{\GL
    72}(\F_{2^7}^\times)$.  One of these transversals has permutation
  monomial $P_7 = X^7$.
\end{prop}
\begin{proof} 
  We proceed as in the proof of Proposition~\ref{prop:PG32}.
\end{proof}

  When $d=6,8,9$,  {\bf no} surface polynomials for $\F_{2^d}$ are
  known and surface monomials do not exist (Lemma~\ref{lemma:transPGN}).

  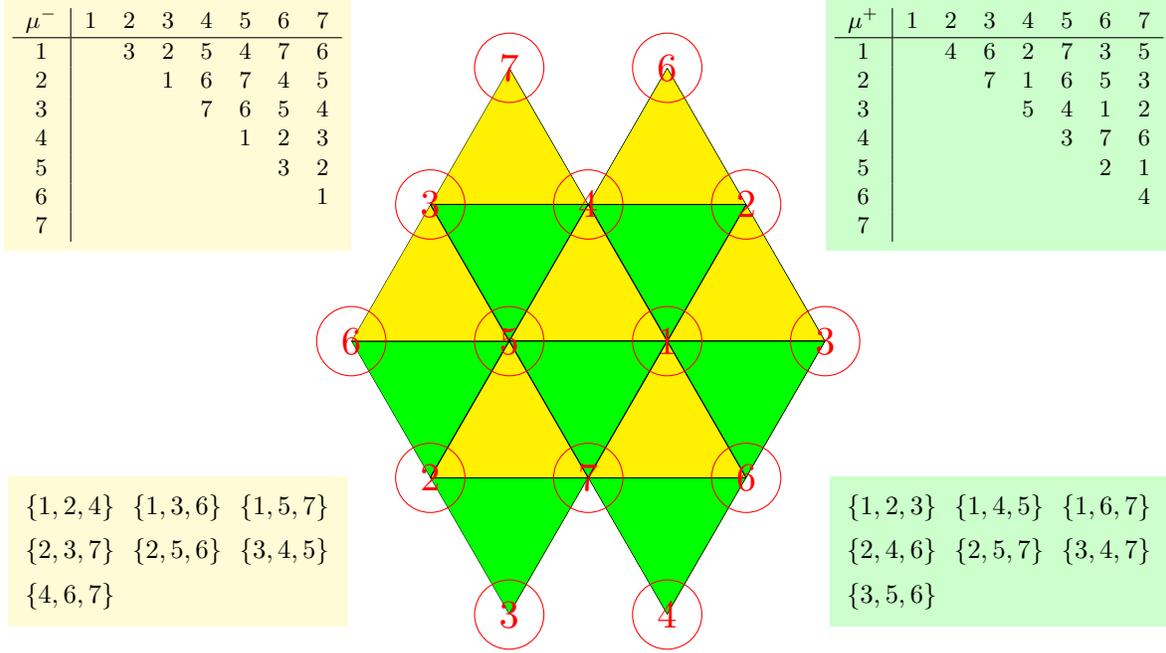
\begin{figure}[t]
  \centering
\begin{tikzpicture}[scale=.6]
 \pgfmathsetmacro{\r}{3.5};
 \pgfmathsetmacro{\a}{360/6};
 
 \begin{scope}
   
\clip (240 : 2*\r) -- ++ (60 : \r) -- ++(300:\r) 
  -- ++(60:2*\r) -- ++(120:2*\r) -- ++(240:\r) -- ++(120:\r)
 -- ++(240:2*\r) -- ++(300:2*\r) --  (240 : 2*\r) -- cycle;

     \foreach \k in {1,3,5,7} {
     \draw[fill=green] 
     (0,0) -- (\k*\a:\r) -- (\k*\a+\a:\r) -- (0,0) --
     cycle;};
       \foreach \k in {2,4,6} {
     \draw[fill=yellow] (0,0) -- (\k*\a:\r) -- (\k*\a+\a:\r) -- (0,0) --
     cycle;};
 
    \begin{scope}[shift={(0:\r)}]
     \foreach \k in {1,3,5,7} {
     \draw[fill=green] 
     (0,0) -- (\k*\a:\r) -- (\k*\a+\a:\r) -- (0,0) --
     cycle;};
       \foreach \k in {2,4,6} {
     \draw[fill=yellow] (0,0) -- (\k*\a:\r) -- (\k*\a+\a:\r) -- (0,0) --
     cycle;};
   \end{scope}

 \begin{scope}[shift={(60:\r)}]
     \foreach \k in {1,3,5,7} {
     \draw[fill=green] 
     (0,0) -- (\k*\a:\r) -- (\k*\a+\a:\r) -- (0,0) --
     cycle;};
       \foreach \k in {2,4,6} {
     \draw[fill=yellow] (0,0) -- (\k*\a:\r) -- (\k*\a+\a:\r) -- (0,0) --
     cycle;};
   \end{scope}

 \begin{scope}[shift={(120:\r)}]
     \foreach \k in {1,3,5,7} {
     \draw[fill=green] 
     (0,0) -- (\k*\a:\r) -- (\k*\a+\a:\r) -- (0,0) --
     cycle;};
       \foreach \k in {2,4,6} {
     \draw[fill=yellow] (0,0) -- (\k*\a:\r) -- (\k*\a+\a:\r) -- (0,0) --
     cycle;};
   \end{scope}
 
 \begin{scope}[shift={(180:\r)}]
     \foreach \k in {1,3,5,7} {
     \draw[fill=green] 
     (0,0) -- (\k*\a:\r) -- (\k*\a+\a:\r) -- (0,0) --
     cycle;};
       \foreach \k in {2,4,6} {
     \draw[fill=yellow] (0,0) -- (\k*\a:\r) -- (\k*\a+\a:\r) -- (0,0) --
     cycle;};
   \end{scope}

 \begin{scope}[shift={(240:\r)}]
     \foreach \k in {1,3,5,7} {
     \draw[fill=green] 
     (0,0) -- (\k*\a:\r) -- (\k*\a+\a:\r) -- (0,0) --
     cycle;};
       \foreach \k in {2,4,6} {
     \draw[fill=yellow] (0,0) -- (\k*\a:\r) -- (\k*\a+\a:\r) -- (0,0) --
     cycle;};
   \end{scope}

 \begin{scope}[shift={(300:\r)}]
     \foreach \k in {1,3,5,7} {
     \draw[fill=green] 
     (0,0) -- (\k*\a:\r) -- (\k*\a+\a:\r) -- (0,0) --
     cycle;};
       \foreach \k in {2,4,6} {
     \draw[fill=yellow] (0,0) -- (\k*\a:\r) -- (\k*\a+\a:\r) -- (0,0) --
     cycle;};
   \end{scope}

\end{scope}

\node[red,scale=1.5,draw, circle] at (0,0) {$1$};
\node[red,scale=1.5,draw, circle] (3a) at (0:\r) { $3$};
\node[red,scale=1.5,draw, circle] at (60:\r) {$2$};
\node[red,scale=1.5,draw, circle] at (120:\r) { $4$};
\node[red,scale=1.5,draw, circle] at (180:\r) { $5$};
\node[red,scale=1.5,draw, circle] at (240:\r) { $7$};
\node[red,scale=1.5,draw, circle] at (300:\r) { $6$};

\node[red,scale=1.5,draw, circle] at ($(60:\r)+(120:\r)$) {$6$};

\node[red,scale=1.5,draw, circle] (7a) at ($(120:\r)+(120:\r)$) {$7$};
\node[red,scale=1.5,draw, circle]  at ($(120:\r)+(180:\r)$) {$3$};

\node[red,scale=1.5,draw, circle] (6a)  at ($(180:\r)+(180:\r)$) {$6$};
\node[red,scale=1.5,draw, circle] at ($(180:\r)+(240:\r)$) {$2$};

\node[red,scale=1.5,draw, circle] at ($(240:\r)+(240:\r)$) {$3$};
\node[red,scale=1.5,draw, circle] at ($(240:\r)+(300:\r)$) {$4$};

\node[anchor= south west,fill=green!20] (mu+) at ($(3a)+(0,2)$) {\small
\begin{tabular}{>{$}c<{$}| *{7} {>{$}c<{$}}}
 \mu^+  & 1 & 2 & 3 & 4 & 5 & 6 & 7 \\ \hline
 1 && 4 & 6 & 2 & 7 & 3 & 5 \\
 2 &&& 7 & 1 & 6 & 5 & 3 \\
 3 &&&& 5 & 4 & 1 & 2 \\
 4 &&&&& 3 & 7 & 6 \\
 5 &&&&&& 2 & 1\\
 6 &&&&&&& 4 \\
 7 &&&&&&&
\end{tabular}};

\node[anchor= south east, fill=yellow!20] (mu-) at ($(6a)+(0,2)$) 
{\small
\begin{tabular}{>{$}c<{$}| *{7} {>{$}c<{$}}}
  \mu^- & 1 & 2 & 3 & 4 & 5 & 6 & 7 \\ \hline
  1 && 3 & 2 & 5 & 4 & 7 & 6 \\
  2 &&& 1&6&7&4&5\\
  3 &&&& 7&6&5&4\\
  4 &&&&& 1&2&3\\
  5 &&&&&& 3&2\\
  6 &&&&&&& 1\\
  7 &&&&&&&
\end{tabular}};

 \matrix  [matrix of math nodes, below = 3cm of mu-,
 anchor=north,fill=yellow!20] {
  \{1,2,4\} & \{1,3,6\} & \{1,5,7\} \\ 
  \{2,3,7\} & \{2,5,6\} & \{3,4,5\} \\
  \{4,6,7\} \\};

 \matrix  [matrix of math nodes, below = 3cm of mu+,
 anchor=north,fill=green!20] {
  \{1,2,3\} & \{1,4,5\} & \{1,6,7\} \\ 
  \{2,4,6\} & \{2,5,7\} & \{3,4,7\} \\
  \{3,5,6\} \\};

\end{tikzpicture}  
  \caption{The orientable Steiner surface
    $(\mathrm{PG}(8),\mathrm{PG}(8)^{P_3})$ of genus $1$}
  \label{fig:stinertria}
\end{figure}


Here is what we know about the $2$-chromatic numbers:
\begin{itemize}
\item For all $d \geq 4$, $3 \leq \chs 2{\mathrm{PG}(2^d)} \leq \chs
  2{\mathrm{PG}(2^{d+1})} \leq \chs 2{\mathrm{PG}(2^d)}+1$ 
   and $\chs 2{\mathrm{PG}(2^d)} < d$
  \cite{rosa69,rosa70}. Furthermore, $\chs 2{\mathrm{PG}(8)} = \chs 2{\mathrm{PG}(16)} = 3$
and  $\chs 2{\mathrm{PG}(32)} = 4$ \cite{rosa70}.

\item It is known from \cite[Theorem 8]{fugere94} that $\chs 2{\mathrm{PG}(64)} = 5$.
  Consequently, 
  $\chs 2{\mathrm{PG}(128)} \geq \chs 2{\mathrm{PG}(64)}= 5$.

\item $\chs 2{\mathrm{PG}(2^d)} \to
  \infty \text{\ for\ } d \to \infty$ by \cite[Corollary 2]{haddad99}.

\item The $2$-chromatic number $\chs 2{\mathrm{PG}(8) \cup
    \mathrm{PG}(8)^{P_3}} = 3$ by direct computer calculations.

\item The $2$-chromatic number $\chs 2{\mathrm{PG}(16) \cup
    \mathrm{PG}(16)^{P_4}} = 3$ by direct computer calculations.

\item The $2$-chromatic number $\chs 2{\mathrm{PG}(32) \cup
    \mathrm{PG}(32)^{P_5} } = 4$ since $\chs 2{\mathrm{PG}(32)} \geq
  4$ and it is possible with a computer to find a $(4,2)$-coloring of
  this genus $126$ nonorientable surface on $31$ vertices.

\item The $2$-chromatic number $\chs 2{\mathrm{PG}(128) \cup
    \mathrm{PG}(128)^{P_7}}$ equals $5$ or $6$ since $\chs
  2{\mathrm{PG}(128)} \geq 5$, and it is possible with a computer to
  find a $(6,2)$-coloring of this genus $2542$ nonorientable surface on $127$ vertices with $f=(127,8001,5334)$.
  This is the smallest known example (in terms of number of faces) of a triangulated surface with
    $2$-chromatic number at least $5$. (For an orientable genus $620$ surface with $f=(2017,9765,6510)$ and $\chi_2= 5$ see Section~\ref{sec:2-2}.)
 A list of facets \texttt{PG128\_PG128P7} of the triangulation $\mathrm{PG}(128) \cup \mathrm{PG}(128)^{P_7}$
can be found online at \cite{BenedettiLutz_LIBRARY}. 
\end{itemize}

\begin{thm}\label{thm:chr2Seq5}
   The $2$-chromatic number of the 
    genus $2542$ nonorientable triangulated surface
     $\mathrm{PG}(128) \cup \mathrm{PG}(128)^{P_7}$ on $127$ vertices
     is $5$ or $6$.
\end{thm}

The authors of \cite{grannell1998} assert on p.~$333$ that
$\mathrm{PG}(32)$ has a \lq cyclic bi-embedding in an orientable
surface\rq .  We have not been able to verify this assertion. However, we go along with Theorem~5 in \cite{rifa2014}
that rules out orientable bi-embeddings of ${\mathrm{PG}(2^d)}$ for $5\leq d\leq 19$.
Presumably, by Corollary~\ref{prop:orbitcat}.\eqref{prop:orbitcat4},
$\GL 52 \cap \GL 52^T$ would contain a Sylow $31$-subgroup of $\GL 52$
for any such transversal $T$, but this contradicts
Proposition~\ref{prop:PG32} according to which
such transversals are nonorientable.

It is claimed in \cite[Theorem 3.1]{donovan2010} that
$\mathrm{PG}(N)$ admits an (orientable?) transversal for all
$N=2^d$. The proof appears to be incorrect.

\section{Chromatic numbers of higher-dimensional manifolds}
\label{sec:colASC}

It was noted in the introduction that the chromatic numbers of a (compact) triangulable $d$-manifold $M^d$ 
form a descending sequence
\begin{equation*}
  \chs 1{M^d} \geq \chs 2{M^d} \geq \cdots \geq \chs d{M^d} \geq \chs{d+1}{M^d} = 1.
\end{equation*}
We shall now see that roughly the first half of these chromatic
numbers are infinite. 
As usual, $B^d$ denotes the $d$-dimensional ball and $S^d$ the
$d$-dimensional sphere.  Then $\chs s{M^d} \geq \chs s{{B}^d}$
for $s \leq d$ and $\chs s{{S}^d} = \chs s{{B}^d}$ for $s<d$.

\begin{lemma}\label{lemma:monotone}
  $\chs s{{S}^d} \leq \chs s{{S}^{d+1}}$ for $d\geq 1$ and $1\leq s\leq d$.
\end{lemma}
\begin{proof}
  Any triangulation of ${S}^d$ is a subcomplex of a
  triangulation of its suspension ${S}^{d+1}$.
\end{proof}

\begin{thm}\label{thm:chrsSd}
  $\chs s{{S}^d} = \infty$ for $d \geq 3$ and $1 \leq s \leq \lceil d/2 \rceil$.
\end{thm}
\begin{proof}
Let $d\geq 3$ be odd. Then
  ${S}^d$ admits a triangulation $\partial \mathrm{CP}(m,d+1)$
  as the boundary of the cyclic polytope $\mathrm{CP}(m,d+1)$ on $m$
  vertices in $\R^{d+1}$. The triangulation $\mathrm{CP}(m,d+1)$ is
  $\lfloor (d+1)/2 \rfloor$-neighborly in the sense that it has the
  same $s$-skeleton as the full $(m-1)$-simplex $\Delta^{m-1}$ when $s < \lfloor
  \frac{d+1}{2} \rfloor = \lceil d/2 \rceil$. Thus $\mathrm{CP}(m,d+1)$ has the same $s$-chromatic number,
  $\lceil m/s \rceil$, as $\Delta^{m-1}$ when $1 \leq s < \frac{d+1}{2}$.
  Since we can choose $m$ to be arbitrarily large, and since $\chs s{{S}^{d-1}} \leq \chs s{{S}^d}$ for even $d$ by Proposition~\ref{lemma:monotone}, the statement follows for $1 \leq s < \lceil d/2 \rceil$.

  Let $s=\lceil d/2 \rceil$. The
  Euclidean $(2s-1)$-space contains $s$-dimensional geometric simplicial complexes with arbitrarily
  high $s$-chromatic numbers \cite[Theorem 22]{HeisePanagiotouPikhurkoTaraz2014}.
  Since we can extend these embedded triangulations to triangulations of $B^{2s-1}$  \cite[Theorem I.2.A]{bing83}, 
  we see that
  $\chs {s}{{S}^{2s-1}} = \infty$.  Then also $\chs {s}{{S}^{2s}} = \infty$ by the monotonicity of
  Lemma~\ref{lemma:monotone}.
\end{proof}

\begin{thm}\label{thm:chsssdmfd}
  $\chs s{{M}^d} = \infty$ when ${M}^d$ is a triangulable $d$-manifold with $d \geq 3$  and $s \leq \lceil
  d/2 \rceil$.
\end{thm}
\begin{proof}
Follows from Theorem~\ref{thm:chrsSd} by taking connected sums of a triangulation ${M}^d$ with triangulations of $S^d$.
\end{proof}

The interesting chromatic numbers of a $d$-manifold $M^d$ of
dimension $d \geq 3$ are
\begin{equation*}
  \chs s{M^d}, \qquad \lceil d/2 \rceil < s \leq d.
\end{equation*}
In particular, for a manifold $M$ of dimension $3$, the only
unknown chromatic number is $\chs 3M$. The most basic question
here is to determine $\chs 3{{S}^3}$.  More generally, it
remains an open problem to determine $\chs d{{S}^d}$ as a
function of the dimension $d \geq 3$.

\section{An explicit example of a triangulated $3$-sphere with $\chi_2 = 5$}
\label{sec:highspheres}

Triangulated $3$-spheres can have arbitrarily large $2$-chromatic number (Theorem~\ref{thm:chrsSd}). 
However, it seems to be hard to construct or find explicit examples of small size with $\chi_2>4$.

The first place to look for concrete examples certainly are boundary complexes of
cyclic $4$-polytopes. Yet, their $2$-chromatic numbers are at most~$3$ 
even though they have arbitrarily large $1$-chromatic numbers.

\begin{exmp}\label{exmp:cyclicpolytope}
  The boundaries $\partial\mathrm{CP}(m,4)$ of cyclic $4$-polytopes on $m \geq 4$ vertices 
  are neighborly triangulated $3$-spheres with $\chs 1 {\partial\mathrm{CP}(m,4)} = m$,
  but 
  \begin{equation*}
    \chs 2{\partial\mathrm{CP}(m,4)} =
    \begin{cases}
      2, & \text{$m$ even,} \\ 3, & \text{$m$ odd.} 
    \end{cases}
  \end{equation*}
\end{exmp}

As before for surfaces in Section~\ref{sec:2-2}, we tried bistellar flips  \cite{BjoernerLutz2000} 
to search through the space of triangulations of~$S^3$, but never found an example with  $\chi_2>4$.
Also, we have not been able to find in the literature a single concrete example of a triangulated $3$-sphere  
not admitting a $(4,2)$-coloring.  At least, the literature does contain sporadic examples of triangulated
$3$-spheres with $2$-chromatic number equal to $4$.

\begin{exmp}\label{exmp:altshuler}
  Altshuler's \lq peculiar\rq\ triangulated $3$-sphere \cite{ALT76} with $f=(10,45,70,35)$ 
  has $2$-chromatic number $\chi_2= 4$ (as we verified with the computer).
\end{exmp}

One reason for why it is hard to find triangulations of $S^3$ with $\chi_2>4$
is that by the previous sections we do not have a small example of a 
triangulated orientable surface with $\chi_2>4$. Thus, if we search for obstructions
to $(4,2)$-colorings, triangulated orientable surfaces that are common as subcomplexes
of $S^3$ will not help.

In~\cite{HeisePanagiotouPikhurkoTaraz2014}, Heise, Panagiotou, Pikhurko and Taraz
provided an inductive geometric construction (this construction was also found independently 
by Jan Kyn\v{c}l  and Josef Cibulka as pointed out to us by Martin Tancer)
to yield $2$-dimensional complexes
with arbitrary high $2$-chromatic numbers. The basic idea of the construction
is as follows. Let $T(r-1)$ be a geometric $2$-dimensional simplicial complex in $\R^3$ 
that has all its vertices on the moment curve $(t,t^2,t^3)$ and that has $f=(f_0,f_1,f_2)$
and $\chi_2(T(r))\geq r-1$. Then we can obtain from $T(r-1)$ a new geometric 
$2$-dimensional simplicial complex $T(r)$ in $\R^3$ 
that again has all its vertices on the moment curve, but now has
$f=(\binom{r}{2}f_0+r,\binom{r}{2}(f_1+1+2f_0),\binom{r}{2}(f_0+f_2))$
and $\chi_2(T(r))\geq r$.

As an abstract simplicial complex, $T(r)$ is obtained from $T(r-1)$ 
by taking $\binom{r}{2}$ copies of $T(r-1)$ `attached' to the $\binom{r}{2}$
edges of a complete graph $K_r$: For every edge $e$ of $K_r$ and the corresponding
copy of $T(r-1)$ a triangle $ev$ is added to the complex $T(r)$ for each vertex $v$
of $T(r-1)$. 

\enlargethispage*{3.5mm}

If we take for $T(2)$ a single triangle with $\chi_2=2$,
then
\begin{itemize}
\item $f(T(2))=(3,3,1)$ and $\chi_2(T(2))=2$,
\item $f(T(3))=(12,30,12)$ and $\chi_3(T(2))=3$,
\item $f(T(4))=(76,330,144)$ and $\chi_3(T(2))=4$,
\item $f(T(5))=(765,4830,2200)$ and $\chi_3(T(2))=5$.
\end{itemize}

By choosing the vertices of $T(r)$ appropriately on the moment curve \cite{HeisePanagiotouPikhurkoTaraz2014},
$T(r)$ is realized as a geometric $2$-dimensional simplicial complex in $\R^3$.
According to \cite[Theorem I.2.A]{bing83} (we are grateful to Karim Adiprasito
for reminding us of this result from PL topology), we can extend $T(r)$ 
to a triangulation of the $3$-dimensional ball $B^3$, which contains $T(r)$ as a subcomplex
and thus has $\chi_2(T(r))\geq r$. However, the resulting triangulations of $B^3$
will be of tremendous size.

In the following, we give a first concrete example
of a non-$(4,2)$-colorable triangulated $3$-sphere
by using a topological version of the geometric construction of \cite{HeisePanagiotouPikhurkoTaraz2014}.


\begin{thm}
There is a non-$(4,2)$-colorable $3$-sphere\, \texttt{\rm non\_4\_2\_colorable}\, with face vector $f=(167,1579,2824,1412)$.
\end{thm}

\begin{proof}
The basic idea for constructing a `small' example of a non-$(4,2)$-colorable $3$-sphere is as follows.
We embed a complete graph $K_5$ in $3$-space and attach tetrahedra to it
\begin{compactitem}
\item so that eventually $K_5$ is embedded in a triangulated $3$-ball (with $166$ vertices)
\item in a way that the attached tetrahedra prevent the $K_5$ to have a monochromatic edge.
\end{compactitem}
We then add to the constructed ball the cone over its boundary
to close the triangulation to a non-$(4,2)$-colorable $3$-sphere (with $167$ vertices).
Indeed, if we can guarantee that in any $2$-coloring of the resulting triangulation
none of the $10$ edges of the complete graph~$K_5$ is allowed to be monochromatic (in four colors),
then all admissible $2$-colorings of the triangulation must have at least five colors.

To each of the $10$ edges of $K_5$ we will attach a (small) $15$-vertex ball $B_{15}$
with the properties that 
\begin{compactitem}
\item $B_{15}$  has a $(4,2)$-coloring, but \emph{no} $(3,2)$-coloring,
\item $B_{15}$  has all its $15$ vertices on its boundary
\item  and the $15$ vertices can be lined up to form a Hamiltonian path on the boundary.
\end{compactitem}
A ball with these properties is obtained from the $16$-vertex
\texttt{double\_trefoil} sphere $S_{16,92}$ of
\cite{BenedettiLutz2013a} by deleting the star of the vertex $7$ and
renaming vertex $16$ to $7$.  The list of facets of the ball $B_{15}$
is given in Table~\ref{tbl:B15}.
\begin{table}[h]
\small\centering
\defaultaddspace=0.6em
\caption{The ball $B_{15}$.}\label{tbl:B15}
\begin{center}\footnotesize
\begin{tabular*}{\linewidth}{@{\extracolsep{\fill}}llllllll@{}}
\toprule
 \addlinespace
  $1\,2\,5\,6$    & $1\,2\,5\,12$   & $1\,2\,6\,12$    & $1\,3\,8\,11$   & $1\,4\,5\,6$    & $1\,4\,5\,16/7$   &  $1\,4\,6\,12$   & $1\,4\,10\,13$  \\
  $1\,4\,10\,16/7$  & $1\,4\,12\,13$   & $1\,5\,12\,13$  & $1\,5\,13\,16/7$  &  $1\,8\,9\,14$   & $1\,8\,10\,14$  & $1\,8\,10\,15$  & $1\,8\,11\,15$  \\
  $1\,9\,11\,15$  & $1\,9\,14\,15$  & $1\,10\,13\,14$ & $1\,10\,15\,16/7$ & $1\,13\,14\,16/7$ & $1\,14\,15\,16/7$ & $2\,3\,4\,13$   & $2\,3\,4\,15$  \\
  $2\,3\,13\,15$  & $2\,4\,8\,16/7$   & $2\,4\,10\,13$  & $2\,4\,10\,16/7$  & $2\,5\,6\,14$   & $2\,5\,12\,14$   & $2\,6\,8\,12$   & $2\,6\,8\,16/7$   \\
   $2\,6\,9\,14$   & $2\,6\,9\,16/7$   & $2\,8\,9\,14$   & $2\,8\,12\,14$  & $2\,9\,10\,16/7$  & $3\,4\,12\,13$ &  $3\,4\,12\,15$  & $3\,5\,6\,14$   \\
  $3\,5\,8\,11$   & $3\,5\,11\,14$  & $3\,6\,9\,14$   & $3\,6\,9\,16/7$  & $3\,9\,12\,13$   & $3\,9\,12\,16/7$  & $3\,9\,13\,15$  & $3\,9\,14\,15$  \\ 
  $3\,12\,15\,16/7$ &  $3\,14\,15\,16/7$   & $4\,5\,8\,16/7$   & $4\,6\,12\,15$  & $5\,8\,11\,13$  & $5\,8\,13\,16/7$  & $5\,11\,12\,13$ & $5\,11\,12\,14$ \\
  $6\,8\,12\,15$  & $6\,8\,13\,15$  & $6\,8\,13\,16/7$ &  $8\,10\,12\,14$ & $8\,10\,12\,15$ & $8\,11\,13\,15$ & $9\,10\,12\,16/7$ & $9\,11\,12\,13$ \\ 
  $9\,11\,13\,15$ & $10\,12\,15\,16/7$ &                 &                 &                 &                 \\
 \addlinespace
\bottomrule
\end{tabular*}
\end{center}
\end{table}

It is easy to check (by an exhaustive computer search) that $B_{15}$ has no $(3,2)$-coloring.
For the boundary of $B_{15}$  and the Hamiltonian path 
14--12--11--9--10--2--13--15--4--8--1--3--5--6--7 (in bold) on the boundary
see Figure~\ref{fig:B15}.
\begin{figure}[t]
\begin{center}
\includegraphics[width=13cm]{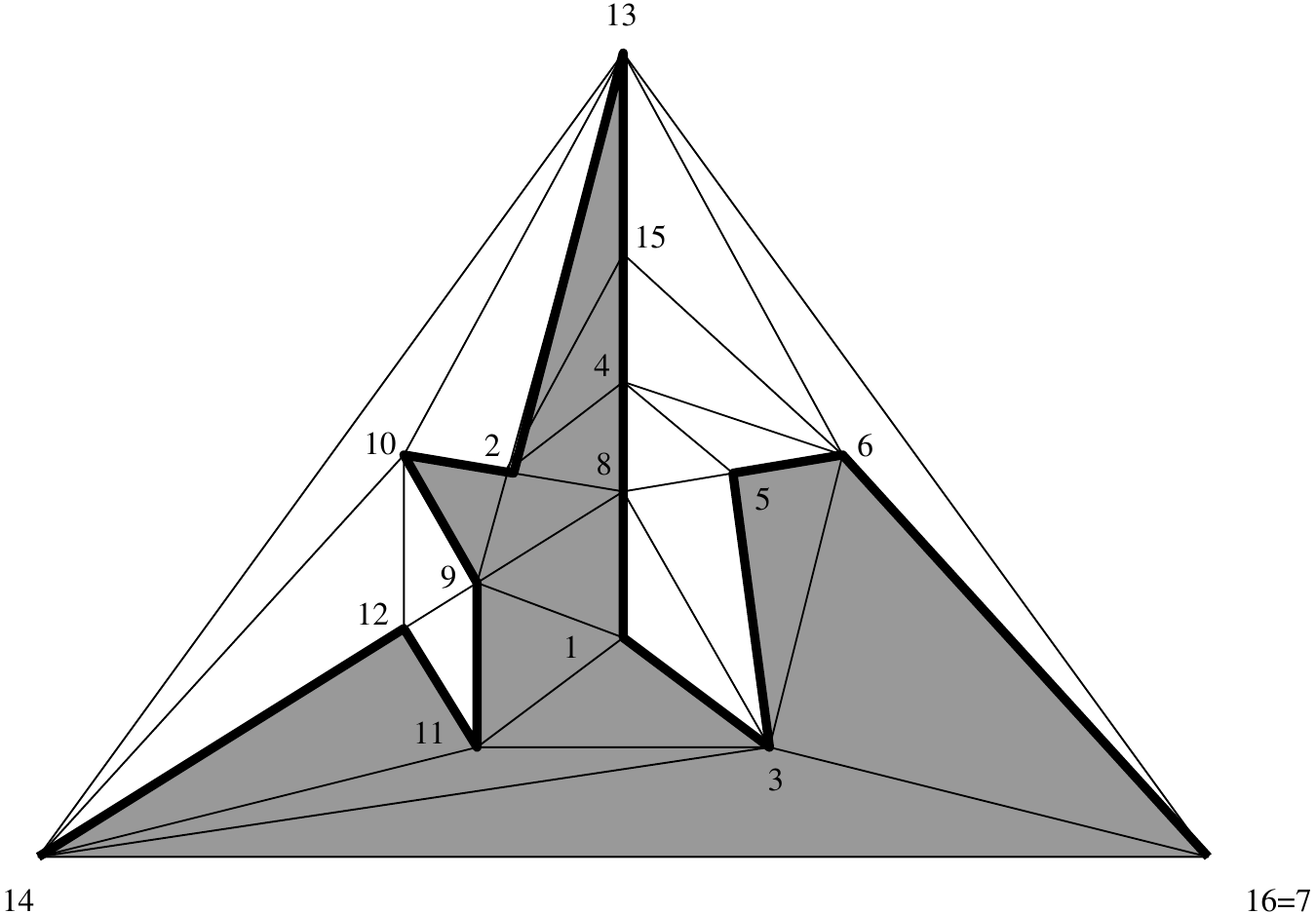}
\end{center}
\caption{The boundary of the $3$-ball $B_{15}$.}
\label{fig:B15}
\end{figure}
In~total, our construction of a non-$(4,2)$-colorable $3$-sphere will have $167$ vertices,
\begin{compactitem}
\item $5$ vertices for the complete graph $K_5$ (vertices $151$--$155$),
\item $10\times 15$ vertices for the $10$ copies of the $15$-vertex ball $B_{15}$ (vertices $1$--$150$),
\item $10$ vertices to glue the $10$ balls to the $10$ edges of $K_5$ (vertices $156$--$165$),
\item  $1$ vertex to close an inner hole (vertex $166$)
\item  and $1$ vertex to cone over the boundary to close the triangulation (vertex $167$).
\end{compactitem}
The construction of the $3$-sphere \texttt{non\_4\_2\_colorable} is in six steps. 

Step I. We embed the complete graph $K_5$ in $3$-space as the $1$-skeleton of a bipyramid over a triangle $153\,154\,155$,
where, in addition, we connect the two apices 151 and 152 by an edge 151--152, as depicted in Figure~\ref{fig:k5}.
\begin{figure}[ht]
\begin{center}
\includegraphics[width=6.25cm]{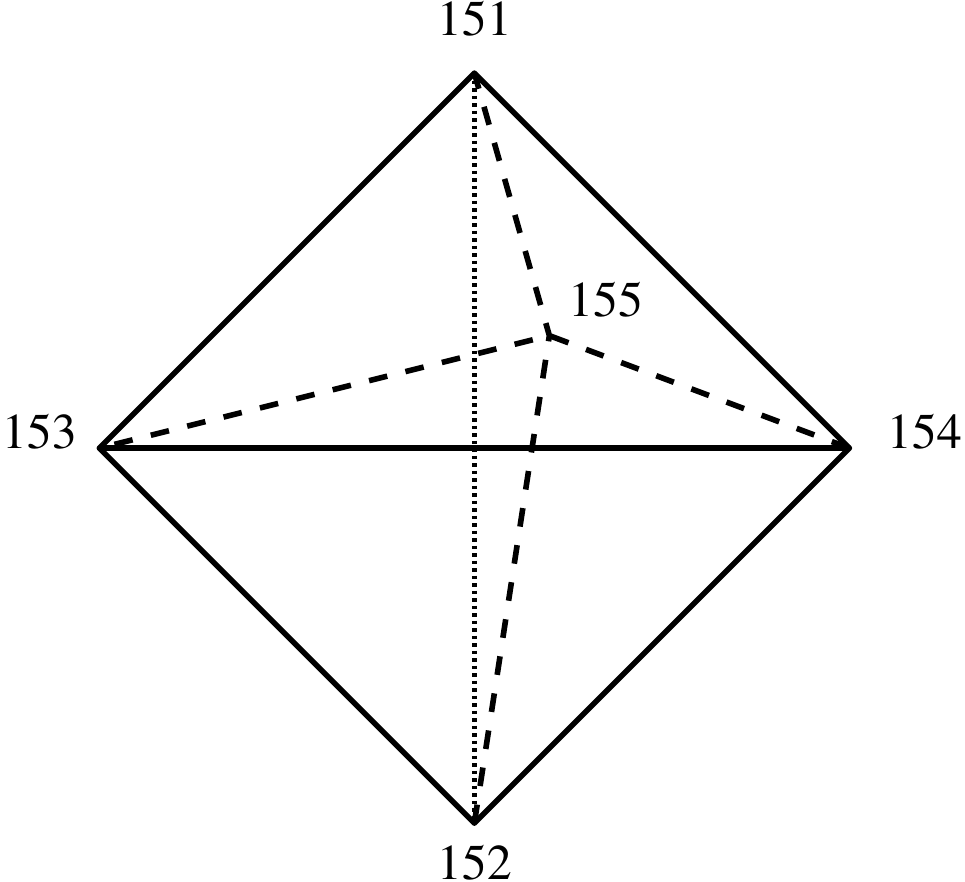}
\end{center}
\caption{The complete graph $K_5$ in $3$-space.}
\label{fig:k5}
\end{figure}

Step II. We add the six triangles of the bipyramid and also two interior tetrahedra $151\,152\,153\,155$   and $151\,152\,154\,155$.
This way, we obtain a mixed $2$- and $3$-dimensional complex that has a single tetrahedral cavity  $151\,152\,153\,154$.
(The two interior tetrahedra and the empty tetrahedron are aligned around the vertical central edge 151--152.)

\begin{figure}[ht]
\begin{center}
\includegraphics[width=14cm]{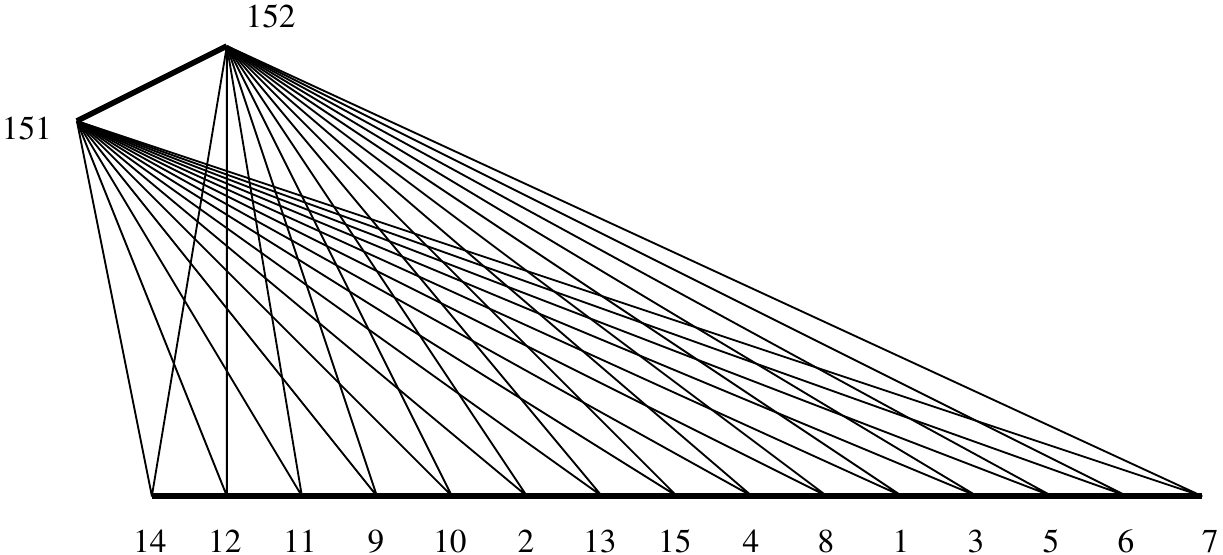}
\end{center}
\caption{The chain of $14$ tetrahedra to attach (a copy of) the ball $B_{15}$ to an edge of $K_5$.}
\label{fig:chain}
\end{figure}

Step III. To the nine edges of $K_5$ that lie on the boundary of the bipyramid, the copies of $B_{15}$ can be attached
so that the attached balls point `outwards', while for the central edge 151--152 the attached copy points into the cavity.
For the attaching itself, we add a chain of $14$ tetrahedra, where each of the tetrahedra is composed by the join
of an edge of $K_5$ and an edge of the Hamiltonian path on the boundary (of a copy) of~$B_{15}$.
In Figure~\ref{fig:chain}, we depict the attaching (via the Hamiltonian path) of the original copy of $B_{15}$ to the edge 151--152.
\emph{It is at this point that we ensure that the respective $K_5$-edge cannot be monochromatic in one of four colors:}
Since $B_{15}$ has no $(3,2)$-coloring, we are forced to use four(!)~colors to color the vertices of the ball.
Each of the vertices of $B_{15}$ appears as a vertex $v$ of the Hamiltonian path on the boundary and is connected 
to the respective $K_5$-edge,
say, 151--152, by a triangle $v\,151\,152$ in the chain of tetrahedra. It follows that the edge 151--152 cannot be
monochromatic in one of four colors, since otherwise one of the triangles in the chain would be monochromatic,
which is not allowed.

Step IV. At this point, we have constructed a mixed $2$- and $3$- dimensional simplicial complex that is embedded 
in $3$-space, but which is not $(4,2)$-colorable. This complex has one cavity that we are going to fill in Step V.
However, we first thicken the attaching ridges (via the Hamiltonian paths) of the $10$ balls. For this, we add to each copy of the ball 
the cone with respect to a new vertex over the grey shaded triangles of Figure~\ref{fig:B15} and, in addition,
over the triangles of one of the sides of the chain of tetrahedra of Figure~\ref{fig:chain}, respectively.

Step V. We fill the cavity by adding the cone with respect to a new vertex $166$ over the triangles that enclose the cavity.
Now, we have obtained a non-$(4,2)$-colorable $3$-ball with $166$ vertices.

Step VI. We add the cone over the boundary of the ball to close the triangulation to a non-$(4,2)$-colorable $3$-sphere with 
$f=(167,1579,2824,1412)$.
\end{proof}

\begin{cor}\label{thm:chr2Seq5}
   The $2$-chromatic number of the example\, \texttt{non\_4\_2\_colorable}\, is $5$.
\end{cor}

\begin{proof}
A $(5,2)$-coloring of the non-$(4,2)$-colorable $3$-sphere \texttt{non\_4\_2\_colorable}
was found with the computer; a~list of facets of the example is available online at \cite{BenedettiLutz_LIBRARY}.
\end{proof}

\begin{rmk}
The example \texttt{non\_4\_2\_colorable} has turned out to be of interest
for testing heuristics for the computations of discrete Morse vectors \cite{BenedettiLutz2014}.
It has perfect discrete Morse vector $(1,0,0,1)$, but this vector is hard to find
due to the $10$ knotted balls that are used in the construction.
\end{rmk}

\section*{Acknowledgements}

We would like to thank Karim Adiprasito and Martin Tancer for valuable
discussions and Brian Brost for assistance in computing chromatic numbers.

\bibliographystyle{amsplain}

\bibliography{.}

\end{document}